\documentclass[11pt,a4paper]{amsart}
\usepackage{amsmath,amsfonts,amsthm,amsopn,color,amssymb,enumitem,mathtools}
\usepackage{palatino}
\usepackage{graphicx}
\usepackage{caption}
\usepackage{subcaption}
\usepackage[colorlinks=true]{hyperref}
\hypersetup{urlcolor=blue, citecolor=red, linkcolor=blue}

\usepackage{geometry}
\usepackage{esint}
\usepackage[title]{appendix}

\numberwithin{equation}{section}

\newtheorem{theorem}{Theorem}[section]
\newtheorem{lemma}[theorem]{Lemma}

\newtheorem{proposition}[theorem]{Proposition}
\newtheorem{remark}[theorem]{Remark}

\newcommand{\e}{\varepsilon}

\newcommand{\R}{\mathbb{R}}

\newcommand{\D}{\mathbb{D}}

 \DeclareMathOperator{\Id}{Id}

\renewcommand{\geq}{\geqslant}

\renewcommand{\L}{\Lambda}

\newcommand{\N}{\mathbb{N}}

\def\bbm[#1]{\mbox{\boldmath $#1$}}
\newcommand{\beq }{\begin{equation}}
	\newcommand{\eeq }{\end{equation}}

\def\sideremark#1{\ifvmode\leavevmode\fi\vadjust{\vbox to0pt{\vss% the remark3
			\hbox to 0pt{\hskip\hsize\hskip1em%                          will appear only
				\vbox{\hsize3cm\tiny\raggedright\pretolerance10000%          on the side
					\noindent #1\hfill}\hss}\vbox to8pt{\vfil}\vss}}}%

\begin{document}
	
	\title[Conformal Metrics on the Disk with Prescribed Curvatures]{Conformal Metrics on the Disk with Prescribed Negative Gaussian Curvature and Boundary Geodesic Curvature}

\author {Rafael L\'opez-Soriano, Francisco J. Reyes-S\'anchez, David Ruiz}

\address{Universidad de Granada \\
	IMAG, Departamento de An\'alisis Matem\'atico \\
		Campus Fuentenueva 	\\	18071 Granada, Spain. }

\

 \email{ralopezs@ugr.es, fjreyes@ugr.es, daruiz@ugr.es}	
	
	\thanks{The authors have been supported by the  MICIN/AEI through
		the Grant PID2024-155314NB-I00, the \emph{IMAG-Maria de Maeztu} Excellence Grant CEX2020-001105-M, and by J. Andalucia via the Research Group FQM-116. F.J.R.S has  been supported by a PhD fellowship (PRE2021-099898) linked to the \emph{IMAG-Maria de Maeztu} Excellence Grant CEX2020-001105-M funded by MICIN/AEI}

	\keywords{Prescribed curvature problem, conformal metric, blow-up analysis, variational methods.}
	
	\subjclass[2020]{35J20, 58J32, 35B44}
	
	%%%%%%%%%%%%%%%%%%%%%%%%%%%%%%%%%%%%%%%%%%%%%%%%%%%%%%%%%%%%%%%%%%%%%%%%%%%%%%%%%%%%%%%%%%%%%%%%%%%%%%%%%%%%%%%%%%%%%%%%%
	%%%%%%%%%%%%%%%%%%%%%%%%%%%%%%%%%%%%%%%%%%%%%%%%%%% Abstract %%%%%%%%%%%%%%%%%%%%%%%%%%%%%%%%%%%%%%%%%%%%%%%%%%%%%%%%%%%%
	%%%%%%%%%%%%%%%%%%%%%%%%%%%%%%%%%%%%%%%%%%%%%%%%%%%%%%%%%%%%%%%%%%%%%%%%%%%%%%%%%%%%%%%%%%%%%%%%%%%%%%%%%%%%%%%%%%%%%%%%%
	\begin{abstract}
    We study the problem of prescribing the Gaussian curvature on the disk and the geodesic curvature on its boundary via a conformal change of the metric. In this paper the case of negative Gaussian curvature is treated, a regime for which the bubbling behavior of approximate solutions is not so well understood. This is due to the possible appearance of blow-up solutions with diverging length and area. We give an existence result under assumptions on the curvatures which are somewhat natural, in view of some obstructions inherent to the problem. Our strategy is variational and relies on the study of certain families of approximated problems. By performing a refined blow-up analysis for solutions with bounded Morse index, we conclude compactness.
\end{abstract}
	
	\maketitle
	%%%%%%%%%%%%%%%%%%%%%%%%%%%%%%%%%%%%%%%%%%%%%%%%%%%%%%%%%%%%%%%%%%%%%%%%%%%%%%%%%%%%%%%%%%%%%%%%%%%%%%%%%%%%%%%%%%%%%%%%%%%%%%%%%%%%%%%%%%%%%%%%%%%%%%%%%%%%%%%%%%%%%%%%%%%%%%%%%%%%%%%%%%%%%%%%%%%%%%%% Section 1 %%%%%%%%%%%%%%%%%%%%%%%%%% %%%%%%%%%%%%%%%%%%%%%%%%%%%%%%%%%%%%%%%%%%%%%%%%%%%%%%%%%%%%%%%%%%%%%%%%%%%%%%%%%%%%%%%%%%%%%%%%%%%%%%%%%%%%%%%%%%%%%%%%%%%%%%%%%%%%%%%%%%%%%%%%%%%%%%%%%%%%%%%%%%%%%%%%%%%%%%

	\section{Introduction}\label{section-introduction}

    Let $(\Sigma, \tilde g)$ be a compact Riemannian surface with smooth boundary $\partial\Sigma$. A metric $g$ is said to be conformal to $\tilde g$ if there exists a smooth positive function $\rho \in C^\infty(\Sigma)$, called the conformal factor, such that $g = \rho\, \tilde g$.
The classical problem of prescribing the Gaussian curvature on $\Sigma$ via a conformal change of the metric was initiated by Berger \cite{Berger1971} and Kazdan-Warner \cite{KazdanWarner1974}, and has attracted the attention of a large amount of work since then. The case $\Sigma = \mathbb{S}^2$ (the so called Nirenberg problem) turns out to be particularly delicate due to the noncompact effect of the group of conformal maps of the sphere, see \cite{ChangGurskyYang1993, ChangYang1987, ChangYang1988, ChenLi95,  Han-DMJ90, Ji2004, StruweDuke}.

When the boundary is not empty, a natural extension is to simultaneously prescribe the geodesic curvature along $\partial\Sigma$. Writing the conformal factor as $\rho = e^u$ for some smooth function $u$, the Gaussian curvature $K$ and the geodesic curvature $h$ of the perturbed metric $g = e^u \tilde g$ are related to the initial curvatures $\tilde K$ and $\tilde h$ through the transformation laws:
\begin{equation*}
    \begin{cases}
        -\Delta u + 2\tilde K = 2K e^{u} & \text{in } \Sigma,\\
        \frac{\partial u}{\partial \nu} + 2\tilde h = 2h e^{u/2} & \text{on } \partial\Sigma,
    \end{cases}
\end{equation*}
where $\Delta$ denotes the Laplace--Beltrami operator associated with $\tilde g$, and $\nu$ is the outward unit normal vector with respect to $\tilde g$.

\medskip
The case of prescribing constant curvatures $K$ and $h$ was treated by Brendle \cite{Brendle2002}, who employed a parabolic flow to obtain solutions in the limit. Classification results for related Liouville-type problems in the half-plane are also available in \cite{GalvezMira2009,LiZhu1995, Zhang2003}. Regarding nonconstant curvatures, the problem was first studied by Cherrier \cite{Cherrier1984}, and later extended in \cite{BattagliaLS2025, LopezSorianoMalchiodiRuiz2019}. In particular, the negative curvature regime $K<0$ has been investigated in depth for multiply connected surfaces in \cite{LopezSorianoMalchiodiRuiz2019}, while more recent advances for arbitrary genus have been obtained through a mean-field approach in \cite{BattagliaLS2025}. For other results in this framework, see \cite{BattagliaReyesSanchez2025, CajuCruzSantos2024}.

\medskip
When $\Sigma$ is the unit disk, the problem becomes the natural boundary counterpart to the classical Nirenberg problem on $\mathbb{S}^2$, and several analogies have been observed in the literature. Observe that by the Riemann mapping theorem we can assume that $\tilde{g}$ is the standard Euclidean metric, so that our problem becomes:
\begin{equation}
\left\{
    \begin{array}{ll}
        -\Delta u = 2K(x)e^{u} & \text{in } \mathbb{D},\\
        \frac{\partial u}{\partial \nu} + 2 = 2h(x)e^{u/2} & \text{on } \partial\mathbb{D}.
    \end{array}\right.
\label{eq:Problem}
\end{equation}
Integrating \eqref{eq:Problem} yields a constraint which corresponds to the Gauss-Bonnet identity:
\begin{equation}
\int_{\mathbb{D}} K(x) e^{u} + \int_{\partial\mathbb{D}} h(x) e^{u/2} = 2\pi.
\label{eq:Gauss-Bonnet}
\end{equation}

Most of the available results in the literature address the problem where at least one of the curvatures to be prescribed is constant (or zero). For instance, the case $h = 0$ was studied in \cite{ChangYang1987} via reflection to a Nirenberg problem on $\mathbb{S}^2$. The zero Gaussian curvature case ($K = 0$) has been extensively investigated, see \cite{ChangLiu1996, DaLioMartinazziRiviere2015,Gehrig2020,  GuoLiu2006,LiLiu2005,LiuHuang2005}. 

When both $K$ and $h$ are nonconstant functions, the problem becomes more intricate, as the interaction of the two curvatures has an effect on the problem that has no counterpart in the classical Nirenberg problem. In \cite{CruzRuiz2018} a new variational approach is developed when both $K$ and $h$ are positive. This variational formulation has been used by Struwe in \cite{Struwe2024} to define a new geometric parabolic flow from which general existence results are derived. On the other hand, in \cite{Ruiz2023} the Leray-Schauder degree of the problem is computed in a compact regime. This last result is based on the blow-up analysis of \cite{JevnikarLopezSorianoMedinaRuiz2022}, that we briefly describe below. Assume that there exists a blow-up sequence of solutions $u_n$ to the problem:
\begin{equation*}
\left\{
    \begin{array}{ll}
        -\Delta u_n = 2K_n(x)e^{u_n} & \text{in } \mathbb{D},\\
        \frac{\partial u_n}{\partial \nu} + 2 = 2h_n(x)e^{{u_n}/2} & \text{on } \partial\mathbb{D},
    \end{array}\right.
\label{eq:Problemsubn}
\end{equation*}
with $K_n \to K$, $h_n \to h$ in $\mathcal{C}^2$ sense, and with bounded area and length, i.e.,
\begin{equation} \label{bdedmass} \int_{\mathbb{D}} e^{u_n} + \int_{\partial\mathbb{D}} e^{{u_n}/2} <C.  \end{equation}

Then there is a unique blow-up point $p \in \partial \D$ such that $h^2(p)+K(p)>0$ and $p$ is a critical point of the auxiliary function
\begin{equation*}\label{eq:Phi}
\Phi(x) = H(x) + \sqrt{H(x)^2 + K(x)},
\end{equation*}
    defined in a neighborhood of $p$, where $H$ denotes the harmonic extension of $h$ to $\mathbb{D}$. In fact, these blow-up solutions have been constructed, see \cite{BattagliaCozziFernandez2023,BattagliaCruzBlazquezPistoia2023,BattagliaMedinaPistoia2021,BattagliaMedinaPistoia2023,CozziFernandez2024}. It is worth pointing out that in both articles \cite{Ruiz2023, Struwe2024}, the behavior of the function $\Phi$ turns out to be crucial.

\medskip

Observe that if both $K$, $h$ are positive, condition \eqref{bdedmass} is automatically satisfied due to \eqref{eq:Gauss-Bonnet}. In contrast, in this paper we consider the case of negative Gaussian curvature. Being more specific, we assume:
\begin{equation}\tag{H} \label{eq:assumption_H} 
K \in \mathcal{C}^2(\overline{\mathbb{D}}), \ h \in \mathcal{C}^2(\partial\mathbb{D}) \ \, \text{and} \  K < 0  \, \text{ in } \overline{\mathbb{D}}. 
\end{equation}

A first step toward the existence of solutions in this setting was given in \cite{LopezSorianoReyesSanchezRuiz2025}, which deals with curvatures $K$, $h$ with some kind of symmetry. Restricting to a space of symmetric functions leads to a significant simplification of the problem, as previously observed in \cite{Moser1973} for the Nirenberg problem. The main purpose of this paper is to remove the symmetry assumptions and give general existence results for negative $K$.

It is worth pointing out that the techniques of \cite{Ruiz2023, Struwe2024} do not apply to the case $K<0$. First, the parabolic flow defined in \cite{Struwe2024} is intrinsically restricted to positive curvatures due to the variational formulation used. Moreover, if $K<0$, the behavior of blow-up solutions is much less understood since \eqref{bdedmass} need not be satisfied; indeed, there are examples showing solutions with diverging area and length, see \cite{LopezSorianoMalchiodiRuiz2019}. Then, the arguments of \cite{Ruiz2023} (which rely on a precise knowledge of the behavior of all possible blow-up solutions) cannot be applied to this framework. 

A satisfactory description of blow-up sequences of solutions without \eqref{bdedmass} was given in \cite{LopezSorianoMalchiodiRuiz2019} under the additional assumption that the Morse index of the solutions remains bounded. This analysis is useful to deal with solutions that are obtained via variational methods, since their Morse index can be controlled by the min-max scheme used. This is the strategy that we follow in this paper to obtain the following existence result.

\begin{theorem}\label{thm:main_existence}
Assume \eqref{eq:assumption_H}, and suppose that $(a)$ and $(b^-)$ or $(b^+)$ are satisfied, where:
\begin{itemize}
    \item[\emph{(a)}] $\displaystyle \min_{x \in \partial\mathbb{D}} h(x) > \max_{x \in \partial\mathbb{D}} \sqrt{|K(x)|}$;

\smallskip
    
    \item[\emph{($b^-$)}] $\partial_\nu \Phi(x) < 0$ for all $x \in \partial\mathbb{D}$ such that $\partial_\tau \Phi(x) = 0$;

    \medskip 
    
    \item[\emph{($b^+$)}] $\partial_\nu \Phi(x) > 0$ for all $x \in \partial\mathbb{D}$ such that $\partial_\tau \Phi(x) = 0$.
\end{itemize}

\medskip 

Then problem \eqref{eq:Problem} admits at least one solution.
\end{theorem}

Here, $\partial_\tau$ and $\partial_\nu$ denote the tangential and outward normal derivatives on the boundary, associated with the unit vectors
$$
  \nu(x)=x=(x_1,x_2) \quad \text{and} \quad \tau(x)=(-x_2,x_1).
$$

\begin{remark} It is worth noting that the assumptions $(a)$ and $(b^{\pm})$ are somewhat necessary, in view of some available obstructions to existence. For instance, if $K=-1$ and $h<1$ is a constant, non-existence of solutions follows from the behavior of horocycles in the hyperbolic plane. This idea has been pushed further in \cite[Theorem 1.2]{LopezSorianoReyesSanchezRuiz2025}, where nonexistence is shown  if  
\begin{equation} \label{nonexist} \min_{x \in \partial\mathbb{D}} h(x) < \max_{x \in \overline{\mathbb{D}}} \sqrt{|K(x)|}. \end{equation} We can easily find examples of $K$ and $h$ satisfying \eqref{nonexist} and $(b^+)$ (or $(b^-)$); hence, assumption $(a)$ is required in Theorem \ref{thm:main_existence}. At this point, let us mention that $(a)$ can be relaxed under condition $(b^-)$, see Theorem \ref{thm:existence_mountain_pass} for details. 

But also $(b^{\pm})$ is necessary in Theorem \ref{thm:main_existence}. To see this, consider $K(x) = -1$ and $h(x) = 3+x_1$. Although \eqref{eq:assumption_H} and $(a)$ are satisfied, \eqref{eq:Problem} does not admit a solution. This comes from the Kazdan-Warner identity derived in \cite[Proposition 2.7]{JevnikarLopezSorianoMedinaRuiz2022}: any solution to \eqref{eq:Problem} satisfies the following:
$$
\int_{\mathbb{D}} e^u \nabla K \cdot F = 4 \int_{\partial\mathbb{D}} \partial_\tau h \, e^{u/2} x_2,
$$
where $F(x_1, x_2):=(1-x_1^2+x_2^2,-2x_1x_2)$. In our case,
$$ \int_{\mathbb{D}} e^u \nabla K \cdot F =0, \ \  \int_{\partial\mathbb{D}} \partial_\tau h \, e^{u/2} x_2 = -\int_{\partial\mathbb{D}} x_2^2 e^{u/2}<0,$$
yielding a contradiction.

Observe that here $H(x)= 3+x_1$ and the equation $\partial_\tau \Phi(p) = 0$ ($p \in \partial \D$) has solutions $p_1=(1,0)$ and $p_2=(-1,0)$. A direct computation gives $\partial_\nu\Phi(p_1) > 0$ while $\partial_\nu\Phi(p_2) < 0$, so that neither $(b^-)$ nor $(b^+)$ is satisfied.  

\end{remark}

As previously mentioned, our proofs use variational methods; solutions to problem \eqref{eq:Problem} are searched as critical points of the energy functional $\mathcal{I}: H^1(\mathbb{D}) \to \mathbb{R}$ defined by
\begin{equation*}\label{eq:EnergyFunctionalDisk}
\mathcal{I}(u) = \int_{\mathbb{D}} \left( \frac{1}{2}|\nabla u|^2 - 2K(x)e^{u} \right) + \int_{\partial\mathbb{D}} \left( 2u - 4h(x)e^{u/2} \right).
\end{equation*}

The geometric properties of $\mathcal{I}$ are far from evident, due to the presence of competing inner terms and boundary terms. Moreover, the possibility of blow-up solutions makes the Palais-Smale condition unlikely to hold. Because of that, we first use a perturbation argument, considering a one-parameter family of functionals
$\mathcal{I}_\varepsilon: H^1(\mathbb{D}) \to \mathbb{R}$ (see \eqref{eq:perturbed_functional_compact}), where $\e$ is a small parameter, not necessarily positive. The basic idea of the proof is to first find critical points of $\mathcal{I}_\e$ via variational methods, and then pass to the limit to get a solution to \eqref{eq:Problem}. Let us elaborate on this procedure.

It turns out that the behavior of the functional $\mathcal{I}_\e$ changes dramatically depending on the sign of $\e$. If $\e \geq 0$ the functional is of mountain-pass type, as low sublevel sets are not path connected. This is not the case for $\e <0$; instead, a more sophisticated 3-dimensional linking structure is found in this case. This linking structure is completely new, and its description uses in an essential way a comparison argument with the constant coefficient case. We consider this construction as one of the primary contributions of our work, and we think that this idea could be useful in other geometric problems of this sort.

The (PS) property seems out of reach even for the perturbed functionals $\mathcal{I}_\e$, as (PS) sequences could eventually diverge in norm. We address this difficulty by using Struwe's monotonicity trick \cite{Jeanjean1999,Struwe1985}, which allows us to find bounded (PS) sequences $v_\e^k$ for almost every $\e$ small. Passing to the limit as $k \to \infty$, solutions are obtained for a family of perturbed problems. Moreover, this procedure can be done controlling the Morse index of the solutions, see \cite{BellazziniRuiz23, BorthwickChangJeanjeanSoave2023}.

At this point, we have obtained the existence of sequences $\e_n \to 0$ (where $\e_n$ can be chosen either positive or negative), and $u_n$ critical points of $\mathcal{I}_{\e_n}$ with bounded Morse index. We can now use \cite{LopezSorianoMalchiodiRuiz2019} to show that the length and area are uniformly bounded, i.e., \eqref{bdedmass} holds. At this  point, we address the compactness of $u_n$ by contradiction, employing the arguments of \cite{JevnikarLopezSorianoMedinaRuiz2022}. In contrast to \cite{JevnikarLopezSorianoMedinaRuiz2022}, the perturbed problems we are working with are not genuinely geometric (i.e., their solutions do not correspond to metrics with prescribed curvatures). Therefore, the study of \cite{JevnikarLopezSorianoMedinaRuiz2022} must be suitably adapted to our setting. Both in the estimate of the error terms, and in the resulting asymptotic expansion, the perturbation plays a significant role in our analysis. In sum, we conclude that if $u_n$ is unbounded, then there exists a unique blow-up point $p \in \partial \D$ with $\partial_{\tau} \Phi(p)=0$ and $\e_n \, \partial_{\nu} \Phi(p)\geq 0$. This contradicts either condition $(b^{-})$ or $(b^{+})$, so that $u_n$ is bounded and we can pass to the limit to obtain a solution to \eqref{eq:Problem}.
\medskip

The remainder of the paper is organized as follows. Section \ref{sec:Preliminaries} introduces the variational framework and presents two auxiliary existence results which, combined, yield Theorem \ref{thm:main_existence}. This section also collects necessary preliminary results. Section \ref{sec:Existence-result} is devoted to the geometric construction of the mountain-pass and linking structures, from which a sequence of approximate solutions is found. In Section \ref{sec:Blow-upAnalysis} we perform a refined blow-up analysis that provides the crucial compactness criterion. With all these ingredients at hand, we conclude the proof of the main results in Section \ref{sec:proofs}. Finally, the Appendix is devoted to an explicit computation of a couple of integral terms appearing in our arguments.

\medskip
\textbf{Notation:} Throughout the paper, the letter $C$ denotes a generic constant, whose value may change from line to line. We also use the standard Landau notation: $O(\varepsilon)$ denotes a quantity bounded by a constant multiple of $\varepsilon$, while $o(\varepsilon)$ denotes a quantity negligible with respect to $\varepsilon$ as $\varepsilon \to 0$.

    %%%%%%%%%%%%%%%%%%%%%%%%%%%%%%%%%%%%%%%%%%%%%%%%%%%%%%%%%%%%%%%%%%%%%%%%%%%%%%%%%%%%%%%%%%%%%%%%%%%%%%%%%%%%%%%%%%%%%%%%%%%%%%%%%%%%%%%%%%%%%%%%%%%%%%%%%%%%%%%%%%%%%%%%%%%%%%%%%%%%%%%%%%%%%%%%%%%%%%%% Section 2 %%%%%%%%%%%%%%%%%%%%%%%%%% %%%%%%%%%%%%%%%%%%%%%%%%%%%%%%%%%%%%%%%%%%%%%%%%%%%%%%%%%%%%%%%%%%%%%%%%%%%%%%%%%%%%%%%%%%%%%%%%%%%%%%%%%%%%%%%%%%%%%%%%%%%%%%%%%%%%%%%%%%%%%%%%%%%%%%%%%%%%%%%%%%%%%%%%%%%%%%
	
	\section{Statement of the Main Results and Preliminaries} \label{sec:Preliminaries}

In order to give a precise statement of our main results, it is convenient to define the following scale-invariant function, introduced in \cite{LopezSorianoMalchiodiRuiz2019}:
\begin{equation*}
\mathfrak{D} : \partial\mathbb{D} \longrightarrow \mathbb{R}, \qquad 
\mathfrak{D}(x) = \frac{h(x)}{\sqrt{|K(x)|}}.
\end{equation*}

We point out that, when assumption $(b^-)$ of 
Theorem \ref{thm:main_existence} is satisfied, hypothesis $(a)$ can be relaxed significantly. Indeed, we will prove the following theorem. 

\begin{theorem}\label{thm:existence_mountain_pass}
Let $K$ and $h$ satisfy \eqref{eq:assumption_H} and assume that the following hypotheses hold:
\begin{itemize}
    \item[\emph{(i)}] There is a point $x_0 \in \partial\D$ with  $\mathfrak{D}(x_0) >  1$;
    \item[\emph{(ii)}] $\partial_\tau \mathfrak{D}(x) \neq 0$ for all $x \in \partial\mathbb{D}$ where $\mathfrak{D}(x) = 1$;
    \item[\emph{(iii)}] $\partial_\nu \Phi(x) < 0$ for all $x \in \partial\mathbb{D}$ where $\partial_\tau \Phi(x) = 0$.
\end{itemize}
Then, problem \eqref{eq:Problem} admits at least one solution.
\end{theorem}

Instead, when assumption $(b^+)$ of Theorem \ref{thm:main_existence} is satisfied, we do need hypothesis $(a)$: 

\begin{theorem}\label{thm:existence_linking}
Let $K$ and $h$ satisfy \eqref{eq:assumption_H} and assume that the following hypotheses hold:
\begin{itemize}
    \item[\emph{(i)}] $\displaystyle{ \min_{x \in \partial\D} h(x) > \max_{x \in \partial\D} \sqrt{|K(x)|}}$;
    \item[\emph{(ii)}] $\partial_\nu \Phi(x) > 0$ for all $x \in \partial\mathbb{D}$ where $\partial_\tau \Phi(x) = 0$.
\end{itemize}
Then, problem \eqref{eq:Problem} admits at least one solution.
\end{theorem}

\subsection{The Perturbation Framework and Variational Settings}

As outlined in the introduction, the functional $\mathcal{I}$ does not seem to satisfy the (PS) condition globally, preventing a direct application of standard critical point theory. To overcome this difficulty, we first employ the monotonicity trick to a family of perturbed functionals $\mathcal{I}_\varepsilon$, finding critical points $u_\varepsilon$. Then, a solution to \eqref{eq:Problem} is found by passing to the limit as $\varepsilon \to 0$. Being more specific, recall that:
\begin{equation}\label{eq:perturbed_functional_compact}   \mathcal{I}_\varepsilon(u) =\frac{\mathcal{I}(u) + \varepsilon \, \mathcal{T}(u)}{1+\varepsilon}, 
\end{equation} 
where $\e$ is a parameter (not necessarily positive) with $|\e|$ small, and $\mathcal{T}: H^1(\mathbb{D}) \to \mathbb{R}^+$ is given by
\begin{equation*}\label{eq:PerturbationTerm}
  \mathcal{T}(u) =\int_{\mathbb{D}} \left( \frac{1}{2}|\nabla u|^2 + e^u - u \right).
\end{equation*}

The normalization factor $(1+\varepsilon)^{-1}$ in the definition of $\mathcal{I}_\varepsilon$ is chosen for notational convenience, to keep the coefficient $1/2$ for the Dirichlet energy term. Substituting the expressions for $\mathcal{I}$ and $\mathcal{T}$, the perturbed functional takes the form:
\begin{equation*}\label{eq:perturbed_functional}
    \mathcal{I}_\varepsilon(u) = \int_{\D} \left( \frac{1}{2}|\nabla u|^2+2\tilde{K}_\varepsilon u - 2K_\varepsilon(x)e^u \right) 
        + \int_{\partial\D} \left( 2\tilde{h}_\varepsilon u - 4h_\varepsilon(x)e^{u/2} \right),
\end{equation*}
with coefficients given by:
\begin{equation}\label{eq:perturbations}
    \tilde{K}_\varepsilon =-\frac{\varepsilon}{2(1+\varepsilon)}, \ \  
    K_\varepsilon(x) =\frac{K(x) - \varepsilon/2}{1+\varepsilon}, \ \ 
    \tilde{h}_\varepsilon =1- \frac{\e}{1+\varepsilon}, \ \  
    h_\varepsilon(x) =\frac{h(x)}{1+\varepsilon}.
\end{equation}

Let us point out that the sign of the perturbation terms, $\tilde{K}_\varepsilon$ and $\tilde{h}_\varepsilon - 1$, is determined by the sign of $\varepsilon$:
$$ \mathrm{sign}(\tilde{K}_\varepsilon) = \mathrm{sign}(\tilde{h}_\varepsilon - 1) = -\mathrm{sign}(\varepsilon). $$

The Euler-Lagrange equation associated with this functional $\mathcal{I}_\varepsilon$ is precisely the perturbed problem:
     \begin{equation}\label{eq:perturbed_problem_general}
        \left\{
        \begin{array}{ll}
            -\Delta u+2 \tilde{K}_\varepsilon=2K_\varepsilon(x)e^{u} &\text{in }\D,\\
            \frac{\partial u}{\partial \nu}+2 \tilde{h}_\varepsilon=2h_\varepsilon(x)e^{u/2} &\text{on }\partial\D.
        \end{array}\right.
    \end{equation}

\subsection{Analytic tools} 
        In order to perform a variational study of the energy functional $ \mathcal{I}_\varepsilon$, we need some preliminary results. We start by revisiting a weak version of the Moser–Trudinger inequality for the boundary term, that will help us to prove that under the assumptions of Theorem \ref{thm:existence_mountain_pass}, in a certain region, the functional is bounded from below.
        
        \begin{lemma}[Lebedev-Milin Inequality, \cite{OsgoodPhillipsSarnak1988}]\label{lem:lebedev_milin}
        For any function $u \in H^1(\D)$, the following trace inequality holds:
        $$
        16\pi\log\left(\int_{\partial\D} e^{u/2}\right) \leq \int_{\D}|\nabla u|^2+ 4\int_{\partial\D}u.
        $$
        \end{lemma}

        In the framework of Theorem \ref{thm:existence_linking}, the preceding inequality does not suffice. Instead, we require an improved version applicable to functions whose mass is spread in several separated regions on the boundary. Such results are known as Chen-Li type inequalities, and the specific version used here follows from \cite[Corollary 2.11]{CruzRuiz2018}.

        \begin{lemma}\label{lem:chen-li}
            Let $l\in\N$ and $\Gamma_1,\Gamma_2,\ldots,\Gamma_l\subset\partial\D$ for which there exists $r>0$ such that $(\Gamma_i)^r\cap(\Gamma_j)^r=\emptyset$ if $i\neq j$, and $\gamma\in(0,\frac{1}{l})$ in such a way that $$\frac{\int_{\Gamma_i}e^{u/2}}{\int_{\partial\D}e^{u/2}}\ge\gamma,\, \forall i=1,\ldots,l,$$
            where $(\Gamma)^r=\left\{ x \in \D \, : \, dist(x,\Gamma)<r \right\}$.
            
            Then, for every $\delta>0$ there exists a constant $C\in\R$, depending on $r$, $\gamma$ and $\delta$, such that $$16\pi l \log\left(\int_{\partial\D}e^{u/2}\right) \leq (1+\delta)\int_{\D}|\nabla u|^2 +8l\int_\D u+C,\quad\forall u\in H^1(\D).$$
        \end{lemma}

With Lemmas \ref{lem:lebedev_milin}, \ref{lem:chen-li} at hand, we will be able to prove the existence of solutions to \eqref{eq:perturbed_problem_general} for some $\e_n \to 0$. In order to prove convergence of those solutions, a blow-up analysis is required. For this, we need Kazdan-Warner identities for problem \eqref{eq:perturbed_problem_general}. These conditions are an extension of those established in \cite[Section 2]{JevnikarLopezSorianoMedinaRuiz2022} to our context. First, let us recall the following Pohozaev-type identity, depending on an arbitrary field $F$.
    
        \begin{lemma}\label{lem:pohosaev}
            Let $u$ be a solution of \eqref{eq:perturbed_problem_general}. Then given any vector field $F:\overline{\D}\to\R^2$, 
            \begin{equation*}
                \begin{split}
                    \int_{\partial\D}&\left[2K_\varepsilon e^u(F\cdot\nu)+(2h_\varepsilon e^{u/2}-2\tilde{h}_\varepsilon)(\nabla u\cdot F)-\frac{|\nabla u|^2}{2}F\cdot\nu\right]=\\
                    &=\int_\D\left[2\tilde{K}_\varepsilon\nabla u\cdot F+2e^u(\nabla K_\varepsilon\cdot F+K_\varepsilon\,\nabla\cdot F)+DF(\nabla u,\nabla u)-\nabla \cdot F\frac{|\nabla u|^2}{2}\right].
                \end{split}
            \end{equation*}
        \end{lemma}
        \begin{proof}
            The proof follows by multiplying \eqref{eq:perturbed_problem_general} by $\nabla u\cdot F$ and integrating by parts; see for instance \cite[Lemma 5.5]{LopezSorianoMalchiodiRuiz2019} or \cite[Lemma 2.6]{JevnikarLopezSorianoMedinaRuiz2022}.
        \end{proof}
        
        From this we can obtain the following Kazdan-Warner identity.
        \begin{proposition}\label{prop:Kazdan-Warner_ind}
             Let $u$ be a solution of \eqref{eq:perturbed_problem_general} then
            \begin{equation*}
                \begin{split}
                \int_{\partial\D}&\left[4x_2\partial_\tau h_\varepsilon e^{u/2}-2x_1(\tilde{h}_\varepsilon-1)u-2x_2\partial_\tau \tilde{h}_\varepsilon u+4x_1\tilde{h}_\varepsilon\right]=\\
                &=\int_\D\left[4x_1\tilde{K}_\varepsilon u-u\nabla\tilde{K}_\varepsilon\cdot F+e^u\nabla K_\varepsilon\cdot F-4x_1\tilde{K}_\varepsilon\right]
            \end{split}
            \end{equation*}
            where $F(x_1, x_2):=(1-x_1^2+x_2^2,-2x_1x_2)$. Note that, if $\tilde{K}_\varepsilon$ and $\tilde{h}_\varepsilon$ are constants, the above expression reduces to
            $$
                \int_{\partial\D}\left[4x_2\partial_\tau h_\varepsilon e^{u/2}-2x_1u(\tilde{h}_\varepsilon-1)\right]=\int_\D\left[4x_1u\tilde{K}_\varepsilon+e^u\nabla K_\varepsilon\cdot F\right].
            $$
        \end{proposition}
	\begin{proof}
	    The idea is to consider the variation along the conformal transformations that keep fixed the point $p=(1,0)$. Using Lemma \ref{lem:pohosaev} and following the same ideas as Proposition 2.7 of \cite{JevnikarLopezSorianoMedinaRuiz2022}, and noting that on $\partial \D$ one has that $F \cdot \nu =0$ and $F=-2x_2\tau$ , we obtain 
            $$
                \int_{\partial\D}-2x_2\,\partial_\tau u\left(h_\varepsilon e^{u/2}-\tilde{h}_\varepsilon\right)=\int_\D\left(\tilde{K}_\varepsilon\nabla u\cdot F+e^u\nabla K_\varepsilon\cdot F-4x_1e^uK_\varepsilon\right).
            $$
            We now multiply \eqref{eq:perturbed_problem_general} by $2x_1$ and integrate to obtain 
            \begin{equation*}
                \begin{split}
                    \int_\D4x_1K_\varepsilon e^u&=\int_\D\left(2x_1(-\Delta u)+4x_1\tilde{K}_\varepsilon\right)=\int_\D 2\partial_{x_1} u-\int_{\partial\D}2x_1\nabla u\cdot\nu+\int_\D4x_1\tilde{K}_\varepsilon\\
                    &=\int_\D2\partial_{x_1} u+\int_{\partial\D}4x_1\tilde{h}_\varepsilon-\int_{\partial\D}4x_1h_\varepsilon e^{u/2}+\int_\D4x_1\tilde{K}_\varepsilon.
                \end{split}
            \end{equation*}
            Observe also that, integrating by parts,
            $$\int_{\partial\D}x_2\partial_\tau u=-\int_{\partial\D}x_1u=-\int_{\D}\partial_{x_1} u.$$
            Putting together
            \begin{equation*}
                \begin{split}
                &\int_{\partial\D}-2x_2\,\partial_\tau u\left(h_\varepsilon e^{u/2}-\tilde{h}_\varepsilon\right)=\\
                &=\int_\D\left(\tilde{K}_\varepsilon\nabla u\cdot F+e^u\nabla K_\varepsilon\cdot F\right)-\int_\D2\partial_{x_1} u-\int_{\partial\D}4x_1\tilde{h}_\varepsilon+\int_{\partial\D}4x_1h_\varepsilon e^{u/2}-\int_\D4x_1\tilde{K}_\varepsilon.
                \end{split}
            \end{equation*}
            Integrating by parts
            \begin{equation*}
                \begin{split}
                \int_{\partial\D}&\left(4x_2\partial_\tau h_\varepsilon e^{u/2}+4x_1h_\varepsilon e^{u/2}-2x_1\tilde{h}_\varepsilon u-2x_2\partial_\tau \tilde{h}_\varepsilon u-4x_1h_\varepsilon e^{u/2}+4x_1\tilde{h}_\varepsilon+2x_1u\right)=\\
                &=\int_\D\left(-u\nabla\tilde{K}_\varepsilon\cdot F+4x_1\tilde{K}_\varepsilon u+e^u\nabla K_\varepsilon\cdot F-4x_1\tilde{K}_\varepsilon\right).
            \end{split}
            \end{equation*}
            Then
            \begin{equation*}
                \begin{split}
                \int_{\partial\D}&\left(4x_2\partial_\tau h_\varepsilon e^{u/2}-2x_1(\tilde{h}_\varepsilon-1)u-2x_2\partial_\tau \tilde{h}_\varepsilon u+4x_1\tilde{h}_\varepsilon\right)=\\
                &=\int_\D\left(4x_1\tilde{K}_\varepsilon u-u\nabla\tilde{K}_\varepsilon\cdot F+e^u\nabla K_\varepsilon\cdot F-4x_1\tilde{K}_\varepsilon\right).
            \end{split}
            \end{equation*}
	\end{proof}

    %%%%%%%%%%%%%%%%%%%%%%%%%%%%%%%%%%%%%%%%%%%%%%%%%%%%%%%%%%%%%%%%%%%%%%%%%%%%%%%%%%%%%%%%%%%%%%%%%%%%%%%%%%%%%%%%%%%%%%%%%%%%%%%%%%%%%%%%%%%%%%%%%%%%%%%%%%%%%%%%%%%%%%%%%%%%%%%%%%%%%%%%%%%%%%%%%%%%%%%% Section 3 %%%%%%%%%%%%%%%%%%%%%%%%%% %%%%%%%%%%%%%%%%%%%%%%%%%%%%%%%%%%%%%%%%%%%%%%%%%%%%%%%%%%%%%%%%%%%%%%%%%%%%%%%%%%%%%%%%%%%%%%%%%%%%%%%%%%%%%%%%%%%%%%%%%%%%%%%%%%%%%%%%%%%%%%%%%%%%%%%%%%%%%%%%%%%%%%%%%%%%%%
	
    \section{Existence of solutions of the perturbed problems} \label{sec:Existence-result}
This section is devoted to the first part of our arguments, namely, the construction of solutions to a family the perturbed problems \eqref{eq:perturbed_problem_general}. We show that, for $|\e|$ sufficiently small, the associated functional $\mathcal{I}_\varepsilon$ possesses a rich variational structure that guarantees the existence of a critical point. As outlined in the introduction, the geometry of the functional depends critically on the sign of the perturbation parameter $\e$. Consequently, our analysis is divided into two distinct regimes: we will show a mountain-pass geometry for positive $\e$, whereas a three-dimensional linking structure is exhibited if $\e<0$. In both scenarios, the (PS) condition seems out of reach, and we will find solutions by implementing Struwe's monotonicity trick. In particular, we will do that keeping control on the Morse index of the solutions.

\begin{theorem}\label{prop:ConvergenceSubsequence}
    Assume \eqref{eq:assumption_H} and that there is a point $x_0 \in \partial\D$ with  $h(x_0) >  \sqrt{|K(x_0)|}$. Then, there exists a sequence $\varepsilon_n \to 0$, $\e_n>0$, and $u_n$ a solution to \eqref{eq:perturbed_problem_general} for $\e=\e_n$, such that $\operatorname{ind}(u_n) \leq 1$.
\end{theorem}

\begin{theorem}\label{prop:ConvergenceSubsequence2}
  Assume \eqref{eq:assumption_H} and that ${\min_{\partial\mathbb{D}} h > \max_{\partial\mathbb{D}} \sqrt{|K|}}$. Then, there exists a sequence $\varepsilon_n \to 0$, $\e_n<0$, and $u_n$ a solution to \eqref{eq:perturbed_problem_general} for $\e=\e_n$, such that $\operatorname{ind}(u_n) \leq 3$.
\end{theorem}

In both theorems above, $\operatorname{ind}(u_n)$ denotes the Morse index of $u_n$, defined as:
	\begin{equation}\label{eq:morseindex}
		\begin{split}
			\operatorname{ind}(u_n)=\max \{ &\dim E, \ E \subset H^1(\mathbb{D}) \text{ vector subspace: } \\
			&\mathcal{I}_{\e_n}''(u_n)|_{E\times E} \text{ is negative definite} \}.
		\end{split}
	\end{equation}

\subsection{Mountain-pass geometry for $\boldsymbol{\varepsilon>0}$}
In this section, we prove that if $\e>0$ the functional $\mathcal{I}_{\e}$ has a mountain-pass geometric structure. The argument essentially follows that developed in \cite{LopezSorianoMalchiodiRuiz2019} and is included here for the sake of completeness.

    \begin{proposition}\label{prop: Mountain-pass}
   Under the assumptions of Theorem \ref{prop:ConvergenceSubsequence}, there exist $\e_0>0$, a set $M\subset H^1(\mathbb{D})$, two points $u_0, u_1 \in H^1(\mathbb{D})$, and a class of admissible paths:
    $$
    \Gamma =\left\{ \eta \in C^0([0,1]; H^1(\mathbb{D})) : \eta(0) = u_0, \eta(1) = u_1 \right\},
    $$
    satisfying the following properties for any $\e \in (0, \e_0)$:
    \begin{enumerate}
        \item[\emph{(i)}] $\mathcal{I}_\varepsilon|_M\geq \overline{C}$ for some constant $\overline{C}$;
        \item[\emph{(ii)}] $\max\{\mathcal{I}_\varepsilon(u_0), \mathcal{I}_\varepsilon(u_1)\} < \inf_M\mathcal{I}_\varepsilon$;
        \item[\emph{(iii)}] $\eta([0,1])\cap M\neq\emptyset$ for any $\eta\in\Gamma$.
    \end{enumerate}

    The min-max level $c$ is then defined by:
    $$
    c_\e =\inf_{\eta \in \Gamma} \max_{t \in [0,1]} \mathcal{I}_\varepsilon(\eta(t)) \geq \inf_M\mathcal{I}_\varepsilon > \max\{\mathcal{I}_\varepsilon(u_0), \mathcal{I}_\varepsilon(u_1)\}.
    $$
\end{proposition}

	The proof requires two preliminary results:
\begin{lemma}
The functional $\mathcal{I}_\varepsilon$ is bounded from below on the set
$$
M = \left\{ u \in H^1(\mathbb{D}) :\int_{\partial\D}e^{u/2}= 1 \right\}.
$$
\end{lemma}
\begin{proof}
For a fixed $\varepsilon > 0$, we have that $\mathcal{I}_\varepsilon(u) \geq \frac{1}{1+\e} \mathcal{I}(u)$ for all $u \in H^1(\D)$, then it is sufficient to show that $\mathcal{I}(u)$ is bounded from below on the set $M$.

Let $u \in M$. We bound the functional by analyzing its components.
First, since $K<0$, the term $-2\int_{\D} K(x)e^u$ is positive. Moreover, setting $
{\mathcal{H}=\max_{\partial\D}h(x)}$, we have:
$$ \int_{\partial\D} h(x) e^{u/2} \leq \mathcal{H} \int_{\partial\D} e^{u/2} = \mathcal{H}.$$ 
Finally, using Lemma \ref{lem:lebedev_milin}, we deduce:
$$\mathcal{I}(u)=\int_{\D} \left( \frac{1}{2}|\nabla u|^2  - 2K(x)e^u \right) + \int_{\partial\D} \left( 2u - 4he^{u/2} \right)\geq -4\mathcal{H}.$$

\end{proof}
    \begin{lemma}\label{lem: bubble}
	 For $\e>0$ sufficiently small, there exists a sequence $u_n\in H^{1}(\mathbb{D})$ such that as $n \to \infty$, 
        $$
\mathcal{I}_\varepsilon(u_n) \longrightarrow -\infty \quad \text{and} \quad \int_{\partial\D} e^{u_n/2}\longrightarrow +\infty.
$$
	\end{lemma}
	
	\begin{proof}
		The proof immediately follows from the estimates given in the Appendix of \cite{LopezSorianoMalchiodiRuiz2019}.
	\end{proof}

    \begin{proof}[Proof of Proposition \ref{prop: Mountain-pass}]
    First, let us evaluate the functional on the constant functions $u = -n$, with $n \in \mathbb{N}$:
    \begin{equation*}
    \begin{split}
    \mathcal{I}_\varepsilon(-n) 
    &= - \int_{\mathbb{D}} 2 \tilde{K}_\varepsilon n - \int_{\mathbb{D}} 2 K_\varepsilon(x) e^{-n} - \int_{\partial\D} 2 \tilde{h}_\varepsilon n - \int_{\partial\D} 4 h_\varepsilon(x) e^{-n/2} \\
    & = -2 n \left( \tilde{K}_\varepsilon \pi + \tilde{h}_\varepsilon 2\pi \right)+o(1) \to -\infty \quad \text{as } n \to +\infty, 
    \end{split}
    \end{equation*}
since $ \tilde{K}_\varepsilon \pi + \tilde{h}_\varepsilon 2\pi= \frac{\pi}{1+\varepsilon} \left(2-\frac{\varepsilon}{2}\right)  >0$ by \eqref{eq:perturbations} with $\varepsilon$ sufficiently small.
    \medskip
    
    Then, choose $u_0 = -n$ such that
    $$
   \int_{\partial\D}e^{u_0/2}< 1 \quad \text{and} \quad \mathcal{I}_\varepsilon(u_0) < \inf_{M} \mathcal{I}_\varepsilon.
    $$
    
    By Lemma \ref{lem: bubble}, there also exists $u_1$ such that
    $$
   \int_{\partial\D}e^{u_1/2}> 1 \quad \text{and} \quad \mathcal{I}_\varepsilon(u_1) < \inf_{M} \mathcal{I}_\varepsilon.
    $$
    
    Observe that for any path $\eta\in\Gamma$, by continuity, there exists $t \in (0,1)$ such that
    $$
    \int_{\partial\D}e^{\eta(t)/2}= 1.
    $$
    
    As a consequence,
    \begin{equation*}
    c =\inf_{\eta \in \Gamma} \max_{t \in [0,1]} \mathcal{I}_\varepsilon(\eta(t)) \geq \inf_{M} \mathcal{I}_\varepsilon > \max\{\mathcal{I}_\varepsilon(u_0), \mathcal{I}_\varepsilon(u_1)\}.
    \end{equation*}
    \end{proof}

\subsection{Higher dimensional linking geometry for $\boldsymbol{\varepsilon<0}$}
In this section, we establish a three-dimensional linking geometry for the $\mathcal{I}_\varepsilon$ for negative values of $\e$.

\begin{proposition}\label{prop:linking}
    Under the assumption of Theorem \ref{prop:ConvergenceSubsequence2}, there exists $\delta_0>0$ such that, for any $\delta \in (0, \delta_0)$, there is a set $M\subset H^1(\mathbb{D})$, a continuous map $\L: \partial B \to H^1(\D)$, where $B$ is the unit ball in $\R^3$, and a class of admissible maps:
    $$ \Gamma =\left\{ \eta \in C^0(\overline{B}; H^1(\D)) : \eta|_{\mathbb{S}^2} = \L \right\}, $$ 
    satisfying the following properties for any $\e \in (-\delta, -\delta/2)$:
    \begin{enumerate}
        \item[\emph{(i)}] $\mathcal{I}_\varepsilon|_M\geq\overline{C}$.
        \item[\emph{(ii)}] $\max \mathcal{I}_\varepsilon|_{\L(\partial B)}<\inf_M\mathcal{I}_\varepsilon$.
        \item[\emph{(iii)}] For all $\eta\in\Gamma$, $\eta(B)\cap M\neq\emptyset$.
    \end{enumerate}

    The min-max level $c$ is then defined by: $$c_\e=\inf_{\eta\in\Gamma}\max_{x\in B}\mathcal{I}_\varepsilon\left(\eta(x)\right)\geq\inf_M\mathcal{I}_\varepsilon>
    \max \mathcal{I}_\varepsilon|_{\L(\partial B)}.$$

\end{proposition}

\begin{remark} As the proof will show, the set $M$ is independent of $\delta$ but the map $\Lambda$ does depend on $\delta$ (not on $\e \in (-\delta, -\delta/2)$).
\end{remark}

\begin{proof}
	
First of all, observe that by addition of a constant $v = u+c$, we pass from problem \eqref{eq:Problem} to an analogous one with rescaled curvatures $\overline{K}_\varepsilon = K_\varepsilon e^{-c}$ and $\overline{h}_\varepsilon = h_\varepsilon e^{-c/2}$. By choosing $c$ conveniently we can assume, without loss of generality, that $ h(x) > 1 \mbox{ and } K(x) \in (-1, 0).$ By taking $\delta_0$ sufficiently small we can find $\mathfrak{h}_0>1$, $r>0$ and $l\in(0,1)$ fixed such that:
\begin{equation} \label{farfrom1}
    h_\e(x) > \mathfrak{h}_0 +r \quad \forall x \in \partial \D, \qquad \  K_\e(x) + 1 >r \  \quad \mbox{ for }|x| > 1-l. 
\end{equation} 
    The proof is divided into four steps.
    \medskip
    \noindent

\textbf{Step 1: Definition of $M$ and proof of (i).}
We first show that there exists a subset $M$ on which the functional $\mathcal{I}_\varepsilon$ is bounded from below. This set is defined via constraints on the boundary mass and barycenter:
$$ M = \left\{ u \in H^1(\D) : A(u)=1,\;B(u) = (0,0)\right\}, $$
where
\begin{equation}\label{eq:operators}
    A(u) =\int_{\partial\D} e^{u/2}, \quad
B(u) =\left( \int_{\partial\D} x_1\,e^{u/2}, \int_{\partial\D} x_2\,e^{u/2} \right).    
\end{equation}

Clearly the functional $\mathcal{I}_\varepsilon$ satisfies the basic inequality:
$$ \mathcal{I}_\varepsilon(u) \geq \frac{1}{2}\int_\D|\nabla u|^2 + 2\tilde{K}_\varepsilon\int_\D u + 2\tilde{h}_\varepsilon\int_{\partial\D} u + C,\quad\text{for all }u\in M. $$
We can rewrite this expression as:
\begin{equation} \label{otra} \mathcal{I}_\varepsilon(u) \geq \frac{1}{2}\int_\D|\nabla u|^2 + 2(\tilde{K}_\varepsilon + 2\tilde{h}_\varepsilon)\int_{\D} u + 4\pi\tilde{h}_\varepsilon\left[\frac{1}{2\pi}\int_{\partial\D} u - \frac{1}{\pi}\int_{\D} u\right] + C. \end{equation}
The difference between the boundary mean and the interior mean can be controlled by a Poincar\'{e}-type inequality:
$$ \left| 4\pi\tilde{h}_\varepsilon\left[\frac{1}{2\pi}\int_{\partial\D} u - \frac{1}{\pi}\int_{\D} u\right] \right| \leq C \left(\int_\D|\nabla u|^2\right)^{1/2}. $$
We now employ the Chen-Li type inequality. For any $u \in M$, the barycenter condition $B(u)=(0,0)$ implies that the assumption of Lemma \ref{lem:chen-li} holds with at least $l=2$ regions. Then,
$$ 0 \leq (1+\delta)\int_{\D}|\nabla u|^2 + 16\int_{\D} u + C \Longrightarrow \int_{\D} u \ge -C - \frac{1+\delta}{16}\int_{\D}|\nabla u|^2. $$
Plugging these inequalities into \eqref{otra} we obtain:
\begin{equation*}
    \begin{split}
        \mathcal{I}_\varepsilon(u) &\geq \left(\frac{1}{2}- \frac{(\tilde{K}_\varepsilon+2\tilde{h}_\varepsilon)(1+\delta)}{8}\right)\int_{\D}|\nabla u|^2 - C \left(\int_\D|\nabla u|^2\right)^{1/2} + C\\
        &\geq\frac{1}{5}\int_\D|\nabla u|^2+C.
    \end{split}
\end{equation*}

\medskip 

\textbf{Step 2: Definition of the map $\L$.} Let $\mathbb{S}^2$ be the unit sphere in $\mathbb{R}^3$. We parameterize a point $y \in \mathbb{S}^2$ using cylindrical coordinates $(p, t)$ where $p \in \mathbb{S}^1$ is the angular coordinate and $t \in [-1,1]$. The parametrization is given by
    $$ y = ( t,\sqrt{1-t^2} p_1, \sqrt{1-t^2} p_2), \quad \text{where } p=(p_1, p_2). $$ 

Then the map $\L: \mathbb{S}^2 \to H^1(\D)$ is defined as $$ \L(y) =\L(p, t) = \gamma_p(\tau(t)), $$ where $\tau(t) = (t+1)/2$ is a linear rescaling from $t \in [-1, 1]$ to $\tau \in [0,1]$. For each $p\in\mathbb{S}^1$, $\gamma_p$ is a path in $H^1(\D)$ connecting two fixed functions, $u_{0}$ and $u_{3}$, via two intermediate functions, $u_{1,p}$ and $u_{2,p}$, which depend on the direction $p$. The path $\gamma_p$ is then defined as the concatenation of three curves (as shown in Figure \ref{fig:placeholder}) structured to connect four test functions $u_0$,  $u_{1,p}$, $u_{2,p}$ and $u_3$.

$$
\gamma_p(\tau) =\left\{
    \begin{array}{ll}
        \gamma_{1, p}(3\tau) & \text{for } \tau \in [0, 1/3], \\
        \gamma_{2, p}(3\tau-1) & \text{for } \tau \in [1/3, 2/3], \\
        \gamma_{3, p}(3\tau-2) & \text{for } \tau \in [2/3, 1].
    \end{array}
\right.
$$

\begin{figure}[h!]
    \centering
    \includegraphics[width=0.8\linewidth]{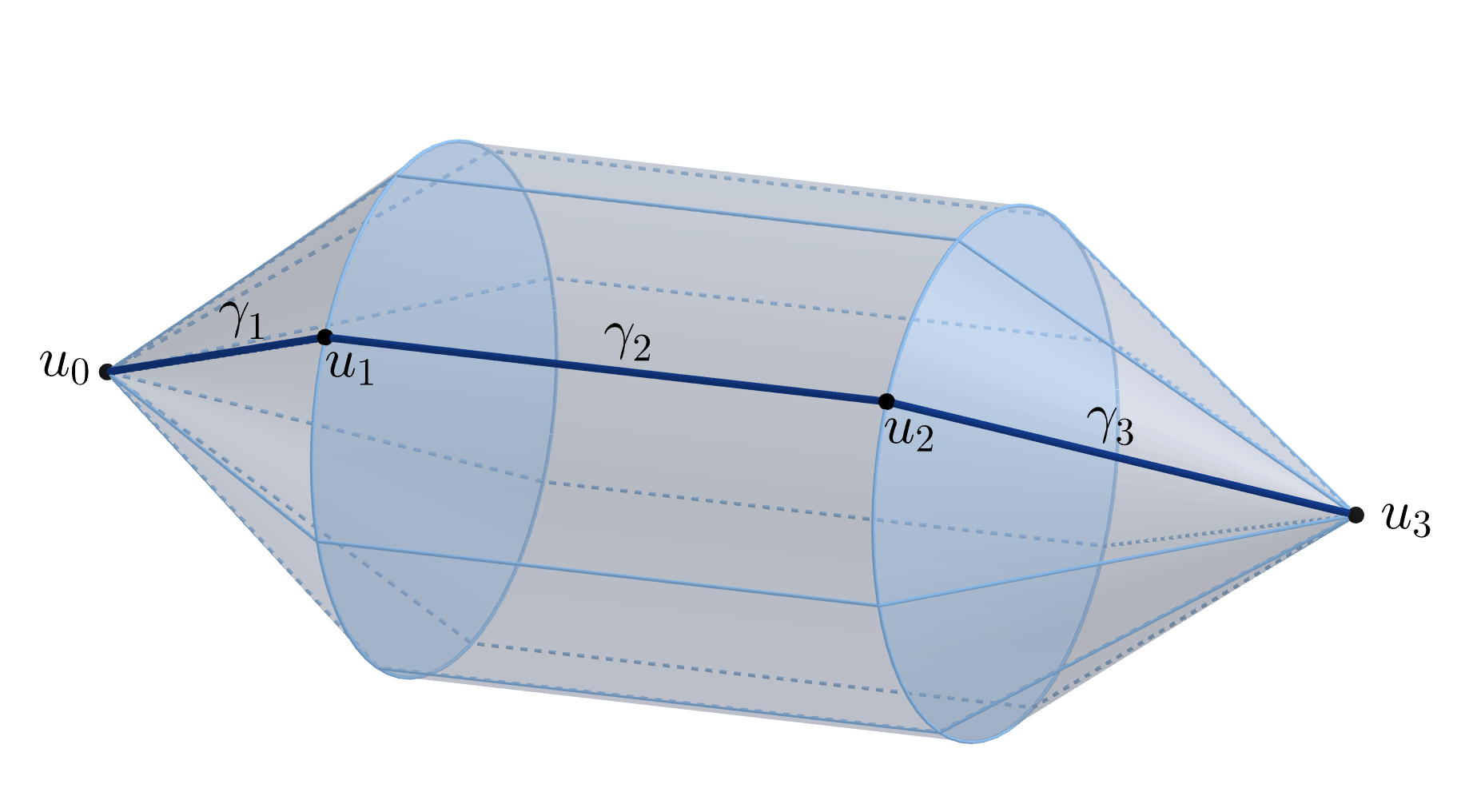}
    \caption{The path $\gamma=\gamma_1+\gamma_2+\gamma_3$ for $p\in\mathbb{S}^1$.}
    \label{fig:placeholder}
\end{figure}

The map $\gamma_p$ is designed to cover regions where the functional $\mathcal{I}_\varepsilon$ takes low values. The whole construction depends on a small parameter $\sigma > 0$, which depends on $\delta$ but not on $\e$. 

We will make use of a family of test functions derived from an auxiliary problem. This problem consists in prescribing constant Gaussian curvature $K_0=-1$ and constant geodesic curvature $\mathfrak{h}$, chosen such that $1 < \mathfrak{h} < \mathfrak{h}_0$. The associated functional is:
\begin{equation}\label{eq:auxiliary_energy}
    \mathcal{J}_\mathfrak{h}(u)=\int_{\D}\left( \frac{1}{2}|\nabla u|^2+2e^u\right) + \int_{\partial\D} \left( 2u-4\mathfrak{h}e^{u/2}\right).
\end{equation}
The Euler-Lagrange equation for $\mathcal{J}_\mathfrak{h}$ is:
\begin{equation*}\label{eq:auxiliary_problem}
	\left\{
	\begin{array}{ll}
		-\Delta u=-2e^{u}&\text{in }\mathbb{D},\\
		\frac{\partial u}{\partial \nu}+2=2\mathfrak{h}e^{u/2}&\text{on }\partial\mathbb{D}.\\
	\end{array}\right.
\end{equation*}

The solutions to this problem form an explicit family, the so-called bubble functions $u_a$, indexed by a point $a\in\D$:
\begin{equation*}\label{eq:bubblefamily}
    u_{a}(x) =2 \log \left( \frac{2 \phi (1 - |a|^2)}{\phi^2 |1 - \overline{a} x|^2 - |x - a|^2} \right), \quad a \in \D,
\end{equation*}
where $\phi =\mathfrak{h} + \sqrt{\mathfrak{h}^2 - 1} > 1$. 

Given $ p \in \mathbb{S}^1$ we will take $a = \lambda p$, $\lambda \in [0,1)$. Observe that, as $\lambda$ tends to 1, the function tends to concentrate around the point $p$. For the sake of clarity we will define the path for the fixed direction $p=(1,0)$, so that $a=(\lambda,0)$. In this case, $u_a(x)$ takes the explicit form:
$$\Psi(x,\phi,\lambda)= 2\log\frac{2(1-\lambda^2)\phi}{\phi^2(1 - 2\lambda x_1 + \lambda^2|x|^2) - (|x|^2 - 2\lambda x_1 + \lambda^2)},$$
where $|x|^2 = x_1^2+x_2^2$. The four test functions are defined as follows:
\begin{equation*}
    \begin{split}
        u_0&=3\log\sigma;\\
        u_1(x)&=\Psi(x,\phi_0,1-\sigma);\\
        u_2(x)&=\Psi(x,1+\sigma,1-\sigma);\\
        u_3(x)&=\Psi(x,1+\sigma,0),
    \end{split}
\end{equation*}
where $\sigma>0$ is a parameter to be chosen and $\phi_0= \mathfrak{h}_0 + \sqrt{\mathfrak{h}_0^2 - 1}>1$.

The first path $\gamma_1$ is simply the segment $$\gamma_1(t) = (1-t)u_0 + tu_1 \quad\text{ for }t\in[0,1].$$ The second curve connects $u_1$ to $u_2$ by continuously varying the parameter $\phi$, while keeping $\lambda = 1-\sigma$ fixed. This path is defined by the transformation
$$\gamma_2(t):=\Psi\big(\cdot , \phi(t), 1-\sigma\big), \quad \text{ where }\,\phi(t)=(1-t)\phi_0+t(1+\sigma), \quad t\in[0,1].$$

The last curve $\gamma_3$ connects $u_2$ to the radial function $u_3$ by varying the concentration parameter $\lambda$ while keeping $\phi = 1+\sigma$ fixed:
$$\gamma_3(t) =\Psi\big(\cdot, 1+\sigma, \lambda(t)\big), \quad \lambda(t) = (1-t)(1-\sigma), \quad t\in[0,1].$$

As will be shown in Step 3, the energy $\mathcal{I}_\varepsilon$ takes low values along the entire path $\gamma=\gamma_1+\gamma_2+\gamma_3$ if $\sigma$ is sufficiently small, depending on $\delta$. 

\medskip 
    \textbf{Step 3: Proof of (ii).} Our objective now is to prove that $\mathcal{I}_\varepsilon$ takes arbitrarily low values on $\L(\partial B)$ if $\sigma>0$ is chosen appropriately. To see this, we take advantage of the conformal invariance of the functional \eqref{eq:auxiliary_energy}. In particular,  
$$\mathcal{J}_{\mathfrak{h}}\left(\Psi(\cdot,\phi,\lambda)\right) = \mathcal{J}_{\mathfrak{h}}\left(\Psi(\cdot,\phi,0)\right)= -8\pi\left(1 + \log\frac{\phi}{2}\right),$$
where the last identity has been computed in \cite[Remark 5.4]{JevnikarLopezSorianoMedinaRuiz2022}.

This motivates us to decompose the perturbed functional $\mathcal{I}_\varepsilon$ as:
$$\mathcal{I}_\varepsilon(u)=\mathcal{J}_\mathfrak{h}(u) + \mathcal{B}_{\varepsilon,\mathfrak{h}}(u),\quad \text{for all }u \in H^1(\D),$$
where $\mathcal{B}_{\varepsilon,\mathfrak{h}}(u)$ is given by:
$$\mathcal{B}_{\varepsilon,\mathfrak{h}}(u)=2\tilde{K}_\varepsilon\int_\D u - 2\int_\D \left(K_\varepsilon(x)+1\right)e^u + 2\left(\tilde{h}_\varepsilon-1\right)\int_{\partial\D} u - 4\int_{\partial\D}\left(h_\varepsilon(x)-\mathfrak{h}\right)e^{u/2}.$$

We now give estimates on  $\mathcal{B}_{\varepsilon,\mathfrak{h}}\left(\Psi(\cdot,\phi,\lambda)\right)$. First, due to the invariance of both exponential terms with respect to $\lambda$, we can apply the estimates established in \cite[Remark 5.4]{JevnikarLopezSorianoMedinaRuiz2022} as:
\begin{equation}\label{eq:estimates_exponential_bubble}
   \int_{\partial\D}e^{\Psi(x,\phi,\lambda)/2}=\int_{\partial\D}e^{\Psi(x,\phi,0)/2}=\frac{2\pi}{\sqrt{\mathfrak{h}^2-1}}, 
\end{equation} and
$$\int_{\D}e^{\Psi(x,\phi,\lambda)}=\int_{\D}e^{\Psi(x,\phi,0)}=2\pi\left(\frac{\mathfrak{h}}{\sqrt{\mathfrak{h}^2-1}}-1\right).$$ 
Using this last identity, combined with condition \eqref{farfrom1}, we can estimate:
\begin{equation} \label{eq:estimates_exponential_bubble_with_K}
    \begin{split}
        \int_\mathbb{D} (K_\varepsilon(x)+1)e^{\Psi} &= \int_{1-l <|x|\leq1} (K_\varepsilon(x)+1)e^{\Psi} + \int_{|x| \leq 1-l} (K_\varepsilon(x)+1)e^{\Psi} \\
        &\geq r\int_{1-l < |x|\leq1}e^{\Psi} + O(1) = r  \int_{\mathbb{D}}e^{\Psi} + O(1) \\
        &= 2\pi r\left(\frac{\mathfrak{h}}{\sqrt{\mathfrak{h}^2-1}}-1\right) + O(1).
    \end{split}
\end{equation}
In the derivation above, the term $O(1)$ refers to the integrals over the set $\{|x| \leq 1-l\}$. Since $\Psi(\cdot,\phi,\lambda)$ remains bounded from above on such set (uniformly on $\phi$, $\lambda$), these contributions are bounded by a constant.

Moreover, by the Jensen's inequality, we can get the following bound:
\begin{equation}\label{eq:jensen}
   \frac{1}{\pi}\int_\D \Psi(x,\phi,\lambda) \leq \, \log\left(\frac{1}{\pi} \int_\D e^{\Psi(x,\phi,\lambda)}\right)=\log\left(2\left(\frac{\mathfrak{h}}{\sqrt{\mathfrak{h}^2-1}}-1\right)\right).
\end{equation}

On the other hand:
\begin{equation}\label{eq:linearboundaryphi}
    \begin{split}
        2\int_{\partial\D}\Psi(x,\phi,\lambda)&=4\int_{\partial\D}\log\frac{2(1-\lambda^2)\phi}{\phi^2\left[1-2\lambda x_1+\lambda^2\right]-\left[1-2\lambda x_1+\lambda^2\right]}\\
        &=8\pi\log\frac{2(1-\lambda^2)\phi}{\phi^2-1}-4\int_{\partial\D}\log(1-2\lambda x_1+\lambda^2)\\
        &=8\pi\log\frac{2(1-\lambda^2)\phi}{\phi^2-1},
    \end{split}
\end{equation}
since $\int_{\partial\D}\log(1-2\lambda x_1+\lambda^2)=0$.

By \eqref{eq:estimates_exponential_bubble}, \eqref{eq:estimates_exponential_bubble_with_K}, \eqref{eq:jensen}, \eqref{eq:linearboundaryphi} and the relations $\frac{\mathfrak{h}}{\sqrt{\mathfrak{h}^2-1}}-1 = \frac{2}{\phi^2-1}$ and $\frac{1}{\sqrt{\mathfrak{h}^2-1}} = \frac{2\phi}{\phi^2-1}$, we get the following estimate:
\begin{equation}
\begin{split}\label{eq:Bound_B}
    \mathcal{B}_{\varepsilon,\mathfrak{h}}(\Psi(x,\phi,\lambda)) & \leq \alpha_\varepsilon\log(1-\lambda^2)-\frac{\beta}{\phi^2-1}+ O\left(\log\left(\frac{1}{\phi^2-1}\right)\right) \\ & \leq \alpha \log(1-\lambda^2)-\frac{\beta}{\phi^2-1}+ O\left(\log\left(\frac{1}{\phi^2-1}\right)\right)
\end{split}
\end{equation}
where 
\begin{align*} %\label{eq:Constant_C_epsilon_phi}
    \alpha_\varepsilon &=8\pi (\tilde{h}_\varepsilon - 1) \geq \alpha=\frac{8 \pi \delta}{2 + \delta};\\
     \beta&=8 \pi r.
\end{align*}
Observe that $\alpha$ and $\beta$ are independent of $\e$.

We are now ready to estimate $\mathcal{I}_\varepsilon|$ on ${\L(\partial B)}$.

\medskip
\noindent
\textit{- The Segment $\gamma_1$.} It can be easily checked that $u_1 \geq u_0$. Recalling that $\gamma_1(t) = (1-t)u_0 + tu_1$, we can estimate
\begin{equation*}
    \begin{split}
        \mathcal{I}_\varepsilon(\gamma_1)&\leq \int_{\D} \left( \frac{1}{2}|\nabla u_1|^2+2\tilde{K}_\varepsilon u_1\right)
        + \int_{\partial\D} 2\tilde{h}_\varepsilon u_1 -  \int_{\D}2K_\varepsilon(x)e^{u_1} - \int_{\partial\D}4h_\varepsilon(x)e^{u_0/2} \\
        &\leq \mathcal{J}_\mathfrak{h}(u_1)+\mathcal{B}_{\varepsilon,\mathfrak{h}}(u_1)+C,
    \end{split}
\end{equation*}
where we have used that $\int_{\partial\D}4h_\varepsilon(x)e^{u_1/2}$ is bounded. 

Then we can apply \eqref{eq:Bound_B} to get 
$$\mathcal{I}_\varepsilon(\gamma_1)\leq  \alpha \log(1-(1-\sigma)^2)-\frac{\beta}{\phi_0^2-1}+ O\left(1 \right)$$
We can conclude by taking $\sigma$ sufficiently small (depending only on $\delta$ and $\phi_0$).

\medskip
\noindent
\textit{- The second curve $\gamma_2$.} Recall that in this curve we use test functions $\Psi(\cdot, \phi, \lambda)$ with $\lambda=1-\sigma$ and $\phi$ taking values from $\phi_0$ to $(1+\sigma)$. We analyze the energy $\mathcal{I}_\varepsilon(\gamma_2(t))$ using the decomposition $\mathcal{I}_\varepsilon= \mathcal{J}_{\mathfrak{h}} + \mathcal{B}_{\varepsilon,\mathfrak{h}}$.
The energy associated with the auxiliary problem is $$\mathcal{J}_{\mathfrak{h}}(\gamma_2(t)) = -8\pi(1 + \log(\phi(t)/2)).$$ Since $\phi(t)$ is continuous and remains within the compact interval $[1+\sigma, \phi_0]$, the term $\mathcal{J}_{\mathfrak{h}}(\gamma_2(t))$ is uniformly bounded for $t \in [0,1]$.

For the perturbation term $\mathcal{B}_{\varepsilon,\mathfrak{h}}$, we apply the upper bound \eqref{eq:Bound_B}:
$$ \mathcal{B}_{\varepsilon,\mathfrak{h}}(\gamma_2(t)) \leq \alpha\log(1 - (1-\sigma)^2)-\frac{\beta}{\phi^2(t)-1}+  O\left(\log\left(\frac{1}{\phi(t)^2-1}\right) \right).$$

We conclude again by taking $\sigma$ sufficiently small.

\medskip
\noindent
\textit{- The third curve $\gamma_3$.} Observe that now $\phi=1+\sigma$ and $\lambda$ takes values between $1-\sigma$ and $0$. Reasoning as above and applying again\eqref{eq:Bound_B}, we obtain:

\begin{equation*}
    \mathcal{B}_{\varepsilon,\mathfrak{h}}(\gamma_3(t)) \leq \alpha \log(1 - \lambda(t)^2) 
    - \frac{\beta}{(1+\sigma)^2-1} 
    +  O\left(\log\left(\frac{1}{(1+\sigma)^2-1}\right) \right),
\end{equation*}
so that we can conclude if $\sigma$ is chosen appropriately.

\medskip
\noindent
\textit{- Conclusion.} The analysis above shows that, given $\delta \in (0, \delta_0)$ and any real number $R$, we can choose $\sigma>0$ small enough so that:
$$
\sup_{\tau}\mathcal{I}_\varepsilon\left(\gamma_p(\tau)\right) < R,
$$ for any $\e \in (- \delta, -\delta/2)$. Since $p \in \mathbb{S}^1$ can be chosen arbitrarily, we conclude the proof of (ii).

\medskip 

\textbf{Step 4: Proof of (iii).}
The final step is the verification of the topological linking condition. We employ a degree theory argument. Let us define the map $\chi: H^1(\D) \to \mathbb{R}^3$ by
\begin{equation*}\label{eq_chi}
   \chi(u) =\left(A(u) - 1, B_1(u), B_2(u)\right),  
\end{equation*}
where $B(u)$ is the barycenter vector defined as in \eqref{eq:operators}. 

By definition, $u \in M \iff \chi(u) = 0$. The linking condition is therefore proved if we show $\deg(\chi \circ \eta, B, 0) \neq 0$. By (ii), $\chi(\L(y)) \neq 0$ for all $y \in \mathbb{S}^2$, so the Brouwer degree $\deg(\chi \circ \L, \mathbb{S}^2, 0)$ is well defined. We claim that this degree is equal to 1.

To prove this, we show that $\chi \circ \L$ is homotopic to the identity map. Consider the standard linear homotopy:
$$H: \mathbb{S}^2 \times [0,1] \to \mathbb{R}^3,\quad H(y, s) =(1-s)\chi(\L(y)) + s y.$$
We show $H(y, s) \neq 0$ for all $(y, s) \in \mathbb{S}^2 \times [0,1]$. The cases $s=0$ and $s=1$ are clear. We proceed by contradiction, assuming $H(y_0, s_0) = 0$ for some $y_0 \in \mathbb{S}^2$ and $s_0 \in (0,1)$. This implies
\begin{equation}\label{eq:invpropor}
    \chi(\L(y_0)) = -\rho \, y_0,
\end{equation} 
for some $\rho = s_0/(1-s_0) > 0$. We distinguish two cases.

\textit{- Case 1:} We consider the poles $y_0=( \pm 1,0,0)$. The function $\L(y_0)$ is either $u_0$ or $u_3$, in both cases $B(u_0)=B(u_3)=(0,0)$ because they are radially symmetric. Furthermore:
\begin{equation*}
\begin{split}
    A(u_0) &= \int_{\partial\D} e^{u_0/2}= 2\pi\sqrt{\sigma^3}\to0,\quad \text{as } \sigma \to 0,\\
    A(u_3) &= \int_{\partial\D} e^{u_3/2} = 2\pi \frac{2(1+\sigma)}{(1+\sigma)^2-1}\to +\infty, \quad \text{as } \sigma \to 0.
\end{split}
\end{equation*}

This implies that $\chi(\L(y_0)) = (c, 0, 0)$ with $c = A(\L(y_0))-1$. In both cases we get a contradiction because $\rho>0$ and $(c,0,0)$ is proportional to $y_0$.

\textit{- Case 2:} We now consider $y_0 = \gamma(\tau_0)$ different from a pole, which can be described via cylindrical coordinates by $(p_0, \tau_0) \in \mathbb{S}^1 \times (0,1)$. We claim that $B\left(\gamma_{p}(\tau)\right)=\kappa p$ for all $p\in\mathbb{S}^1$, $ \tau \in (0,1)$, and for some choice of $\kappa>0$. By rotational invariance, it suffices to analyze the case $p=(1,0)$.

Observe that the component $B_2 = \int_{\partial\D} x_2 e^{\gamma(\tau_0)/2} ds$ vanishes by odd symmetry. In order to estimate $B_1 = \int_{\partial\D} x_1 e^{\gamma(\tau_0)/2} ds$, we split the integral over the right and left half-circles, $\partial\D^+ = \{x\in\partial\D:x_1>0\}$ and $\partial\D^- = \{x\in\partial\D:x_1<0\}$. Applying the change of variables $\tilde{x}= (-x_1, x_2)$ to the integral over $\partial\D^-$, we can write:
$$
B_1 = \int_{\partial\D^+} x_1 \left[ e^{\gamma(\tau_0)(x))/2}- e^{\gamma(\tau_0)(\tilde{x})/2}\right].
$$
We claim that $\gamma(\tau_0)(x) > \gamma(\tau_0)(\tilde{x})$ for $x \in  \partial\D^+$. Note that, for $x\in\partial\D$ we have the strict monotonicity in $x_1$:
\begin{equation*}
    \begin{split}
        \Psi\left(x,\phi,\lambda\right)&=2\log\frac{2(1-\lambda^2)\phi}{(\phi^2-1)(1 - 2\lambda x_1 + \lambda^2) }\\
        &> 2\log\frac{2(1-\lambda^2)\phi}{(\phi^2-1)(1 + 2\lambda x_1 + \lambda^2) }\\
        &=\Psi\left(\tilde{x},\phi,\lambda\right) \quad \text{if } x_1 > 0 
    \end{split}
\end{equation*}
Hence an analogous inequality holds for $u_1$, $\gamma_2(t)$, and $\gamma_3(t)$. It is also inherited by the convex combination $\gamma_1(t) = (1-t)u_0 + tu_1$, since $u_0$ is constant. Thus, for any $\tau_0 \in (0,1)$, $\gamma(\tau_0)(x) > \gamma(\tau_0)(\tilde{x})$, which implies $B_1>0$.

Coming back to equality \eqref{eq:invpropor}, we compare the barycenter components:
$$
\kappa \, p_0 = -\hat{\rho} \, p_0,
$$
but this is a contradiction, as $\kappa > 0$, $\hat{\rho} > 0$.

\medskip

Since the assumption $H(y_0, s_0) = 0$ leads to a contradiction in all cases, $H$ is a valid homotopy and then $\deg(\chi \circ \L, \mathbb{S}^2, 0) = \deg(\Id, \mathbb{S}^2, 0) = 1$. As a consequence $\deg(\chi \circ \eta, \mathbb{S}^2, 0) = 1$ for any $\eta \in \Gamma$, which guarantees the existence of a point $y_0 \in \overline{B}$ such that $\chi(\eta(y_0)) = 0$. This implies $\eta(y_0) \in M$, completing the proof of (iii).
\end{proof}

\subsection{Existence of a Critical Value}

In this subsection we complete the proofs of Theorems \ref{prop:ConvergenceSubsequence} and \ref{prop:ConvergenceSubsequence2}.

\begin{proof}[Proof of Theorem \ref{prop:ConvergenceSubsequence}]
Let $\varepsilon_0>0$ be as in Proposition \ref{prop: Mountain-pass}. We recall the definition of the perturbed functionals $\mathcal{I}_\varepsilon : H^1(\D) \to \R$ for  $\e \in (0, \e_0)$:
$$ \mathcal{I}_\varepsilon(u) = \frac{\mathcal{I}(u) + \varepsilon \, \mathcal{T}(u)}{1+\varepsilon}, \ \mbox{where }\  \mathcal{T}(u) = \int_{\mathbb{D}} \left( \frac{1}{2}|\nabla u|^2 + e^u - u \right).$$ As established in the previous sections, this functional $\mathcal{I}_\varepsilon$ possesses the required min-max geometry (either mountain-pass or higher dimensional linking), the nature of which depends on the sign of $\varepsilon$. 

Let $c_\varepsilon$ denote the corresponding min-max level for $\mathcal{I}_\varepsilon$. The term $\mathcal{T}$ is coercive and strictly positive, and the derivatives $\mathcal{I}'_\varepsilon$, $\mathcal{I}''_\varepsilon$ are uniformly Hölder continuous on bounded sets. We may therefore apply to $\mathcal{I} + \varepsilon \mathcal{T}$ the monotonicity trick of Struwe (\cite{Struwe1985, Jeanjean1999}). In our case, we make use of the abstract results from \cite[Proposition 4.2]{BellazziniRuiz23} or \cite[Theorem 1]{BorthwickChangJeanjeanSoave-p}, which allow us to keep control on the Morse index of the solutions. We obtain the existence of a bounded Palais–Smale sequence $\{v_k^{\varepsilon}\}$ for $\mathcal{I}_\varepsilon$ at the level $c_\varepsilon$ for almost every $\varepsilon\in (0, \e_0)$. Then, up to a subsequence, we may assume that $v_k^{\varepsilon} \rightharpoonup u_\varepsilon$ weakly in $H^1(\mathbb{D})$. By standard regularity arguments, $u_\e$ is a classical solution of the perturbed problem \eqref{eq:perturbed_problem_general} with the coefficients given in \eqref{eq:perturbations}. In particular, one can take a sequence $\e_n>0$ converging to $0$ with that property. 

Furthermore, the solutions obtained via these min-max procedures have a Morse index controlled by the dimension of the underlying geometric structure. In this case, since we are using a mountain-pass argument,
$$ \mathrm{ind}(u_{\varepsilon}) \leq 1. $$
\end{proof}

\begin{proof}[Proof of Theorem \ref{prop:ConvergenceSubsequence2}]

We reason analogously, with only one technical detail. Let $\delta_0$ as given by using Proposition \ref{prop:linking}, and fix $\delta \in (0, \delta_0)$. We consider again the family of functionals $\mathcal{I}_\varepsilon$, with $\e \in (-\delta, -\delta/2)$. Taking into account the linking structure proved in Proposition \ref{prop:linking} (which is independent of $\e$, although it depends on $\delta$), we can argue as above to find $u_\e$ solutions to \eqref{eq:perturbed_problem_general} for almost every $\e \in (-\delta, -\delta/2)$. Since $\delta>0$ is arbitrary, we obtain the desired sequence of solutions.

Since our min-max scheme is three-dimensional in this case, we obtain the Morse index bound:
$$ \mathrm{ind}(u_{\varepsilon}) \leq 3. $$

\end{proof}

%%%%%%%%%%%%%%%%%%%%%%%%%%%%%%%%%%%%%%%%%%%%%%%%%%%%%%%%%%%%%%%%%%%%%%%%%%%%%%%%%%%%%%%%%%%%%%%%%%%%%%%%%%%%%%%%%%%%%%%%%%%%%%%%%%%%%%%%%%%%%%%%%%%%%%%%%%%%%%%%%%%%%%%%%%%%%%%%%%%%%%%%%%%%%%%%%%%%%%%% Section 4 %%%%%%%%%%%%%%%%%%%%%%%%%% %%%%%%%%%%%%%%%%%%%%%%%%%%%%%%%%%%%%%%%%%%%%%%%%%%%%%%%%%%%%%%%%%%%%%%%%%%%%%%%%%%%%%%%%%%%%%%%%%%%%%%%%%%%%%%%%%%%%%%%%%%%%%%%%%%%%%%%%%%%%%%%%%%%%%%%%%%%%%%%%%%%%%%%%%%%%%%
	
\section{Blow-up analysis} \label{sec:Blow-upAnalysis}

In this section, we analyze the properties of blow-up solutions to the perturbed problem
\begin{equation}\label{ecua-compact}
    \begin{cases}
        - \Delta u_n + 2 \tilde{K}_n = 2 K_n(x) e^{u_n} & \text{in } \mathbb{D}, \\
        \frac{\partial u_n}{\partial \nu} + 2 \tilde{h}_n = 2 h_n(x) e^{u_n/2} & \text{on } \partial \mathbb{D},
    \end{cases}
\end{equation}
where the coefficients depend on a sequence $\varepsilon_n \to 0$ according to
\begin{equation}\label{eq:perturbations_n}
    \tilde{K}_{n} =-\frac{\varepsilon_n}{2(1+\varepsilon_n)}, \quad 
    K_{n}(x) =\frac{K(x) - \varepsilon_n/2}{1+\varepsilon_n}, \quad 
    \tilde{h}_{n} =\frac{1}{1+\varepsilon_n}, \quad 
    h_{n}(x) =\frac{h(x)}{1+\varepsilon_n}.
\end{equation}

Most of the related literature on this kind of problems imposes a uniform bounded mass condition in the form: 

\begin{equation}\label{eq:bounded_mass}
	\int_{\mathbb{D}} e^{u_n} + \int_{\partial \mathbb{D}} e^{u_n/2} < C,
\end{equation}
see for instance \cite{BattagliaLS2020,JevnikarLopezSorianoMedinaRuiz2022}. However, in general this condition need not be  satisfied, and the blow-up analysis available in literature is much less complete without \eqref{eq:bounded_mass}. However, more is known if the solutions have bounded Morse index, as happens in our setting.

Following the classical approach, we define the singular set $S$ by
$$
    S =\left\{ x \in \overline{\mathbb{D}} : \exists\, x_n \to x \text{ such that } u_n(x_n) \mbox{ is unbounded from above} \right\}.
$$

The next theorem encompasses the blow-up analysis of \cite{LopezSorianoMalchiodiRuiz2019} (which treats the case of a general surface, and possibly unbounded mass) and \cite{JevnikarLopezSorianoMedinaRuiz2022} (which is concerned with the case of the disk and the localization of the blow-up point). 

\begin{theorem}[Theorem 1.1 of \cite{JevnikarLopezSorianoMedinaRuiz2022} and Theorem 1.4 of \cite{LopezSorianoMalchiodiRuiz2019}]\label{thm: Blowup}
    Assume \eqref{eq:assumption_H}, and let $u_n$ be a sequence of solutions to \eqref{ecua-compact}, where the coefficients are given by \eqref{eq:perturbations_n}, such that $\sup u_n \to +\infty$. Assume also that the Morse index of such solutions is bounded (see \eqref{eq:morseindex}), i.e., $ind(u_n) \leq m$ for some $m \in \mathbb{N}$. 
    
    Then, passing to a subsequence if necessary, we have that 
    the singular set $S$ lies on the boundary $\partial \mathbb{D}$ and admits the decomposition $S = S_0 \cup S_1$, where:
    \begin{equation*}
        \begin{split}
            S_0 &\subset \left\{ x \in \partial \mathbb{D} : \mathfrak{D}(x) = 1, \quad \partial_\tau\mathfrak{D}(x) = 0 \right\},\\
            S_1 &= \{ p_1, \ldots, p_j \} \subset \left\{ x \in \partial \mathbb{D} : \mathfrak{D}(x) > 1 \right\}, \quad \text{with } j \leq m.
        \end{split}
    \end{equation*}
    Moreover, if $S_0 =\emptyset$, then \eqref{eq:bounded_mass} is satisfied, and $S=S_1= \{p\}$, for some $p \in \partial \D$. In this case, there exists a sequence $a_n \in \D$, $a_n \to p$ such that:
        $$
            u_n(x) = u_{a_n}(x) + \psi_n(x),\quad
            u_{a_n}(x) =2 \log \left( \frac{2 \hat{\phi}_n (1 - |a_n|^2)}{\hat{\phi}_n^2 |1 - \overline{a}_n x|^2 + \hat{k}_n |x - a_n|^2} \right),
        $$
        where the error term $\psi_n$ is uniformly bounded. The coefficients above are given by:
        \begin{equation}\label{eq:def_hat_phi_and_k}
            \hat{\phi}_n =\phi_n\left(\frac{a_n}{|a_n|}\right), \quad \hat{k}_n =K_n\left(\frac{a_n}{|a_n|}\right), \quad \hat{h}_n =h_n\left(\frac{a_n}{|a_n|}\right),
        \end{equation}
        with the auxiliary function defined as:
        \begin{equation*}
            \phi_n(x) =h_n(x) + \sqrt{h_n^2(x) + K_n(x)}.
        \end{equation*}
\end{theorem}

The remainder of this section focuses on the bounded mass regime \eqref{eq:bounded_mass}. Our main result establishes the following necessary conditions concerning the localization of the blow-up point:

\begin{theorem}\label{thm:normal_condition}
    Under the assumptions of Theorem \ref{thm: Blowup}, suppose that $S_0= \emptyset$  and let $p \in \partial \mathbb{D}$ be the unique blow-up point. Assume also that either $\e_n>0$ or $\e_n <0$ for all $n \in \N$, Then, $ \partial_\tau \Phi(p) = 0$ and: 
    \[ \mbox{ If } \e_n>0, \mbox{ then }  \partial_\nu \Phi(p) \geq 0; \ \mbox{ instead, if }\e_n<0, \mbox{  then } \partial_\nu \Phi(p) \leq 0 \]

\end{theorem}

The proof of Theorem \ref{thm:normal_condition} requires a refined analysis of the blow-up profile. As a first step, we improve the control on the error term $\psi_n$. Although Theorem \ref{thm: Blowup} ensures that $\psi_n$ is uniformly bounded, the following proposition establishes the sharper quantitative estimate necessary for our argument.

\begin{proposition}\label{prop:v_n_properties}
    Let $\psi_n$ be the error term given in Theorem \ref{thm: Blowup}. Then, for any $\alpha \in (0, 1/2)$, the following bound holds:
    \begin{equation}\label{eq:norm_psi}
        \|\psi_n\|_{C^{0,\alpha}} \leq C (1-|a_n|)^{-2\alpha} \left( 1-|a_n| + |\varepsilon_n| \right),
    \end{equation} 
    where $C > 0$ is a constant independent of $n$.
\end{proposition}

\begin{proof}
    We define:
    $$v_n(x) =u_n(f_{a_n}(x)) + 2\log|f'_{a_n}(x)|$$ 
    with $f_{a}(x) = \frac{a+x}{1+\overline{a}x}$. Invoking \cite[Proposition 4.1]{JevnikarLopezSorianoMedinaRuiz2022}, the following properties hold:
    \begin{itemize}
        \item[\emph{(a)}] $v_n$ solves the problem:
            \begin{equation}\label{eq:Problem_v_n}
                \begin{cases}
                    - \Delta v_n + 2 \tilde{K}_n|f'_{a_n}(x)|^2 = 2 K_n(f_{a_n}(x)) e^{v_n} & \text{in } \mathbb{D}, \\
                    \frac{\partial v_n}{\partial \nu} + 2\left(|f'_{a_n}(x)|(\tilde{h}_n-1)+1\right) = 2 h_n(f_{a_n}(x)) e^{v_n/2} & \text{on } \partial \mathbb{D}.
                \end{cases}
            \end{equation}
        \item[\emph{(b)}] $ \int_{\mathbb{D}} x_1e^{v_n}  = \int_{\mathbb{D}} x_2e^{v_n}  = 0$.
        \item[\emph{(c)}] $a_n \to p$, where $p \in \partial\mathbb{D}$ is the blow-up point.
        \item[\emph{(d)}] $v_n$ is uniformly bounded.
    \end{itemize}

    In view of \emph{(d)}, the nonlinear terms in \eqref{eq:Problem_v_n}, namely $2 K_n(f_{a_n}) e^{v_n}$ and $2 h_n(f_{a_n}) e^{v_n/2}$, are uniformly bounded in $L^\infty(\mathbb{D})$. Moreover, since $f_{a_n}$ is a conformal map, we have the identities:
    \begin{equation} \label{mass} \int_{\mathbb{D}} |f'_{a_n}(x)|^2 = \pi \quad \text{and} \quad \int_{\partial\mathbb{D}} |f'_{a_n}(x)|= 2\pi.\end{equation}
    Consequently, all terms involved in problem \eqref{eq:Problem_v_n} are uniformly bounded in $L^1(\mathbb{D})$ and $L^1(\partial\mathbb{D})$, respectively. Here and in what follows, let us fix $q \in (1,2)$. Standard elliptic regularity ensures that $v_n$ is bounded in $W^{1,q}(\mathbb{D})$ for any $q<2$. Therefore, up to a subsequence, $v_n \rightharpoonup v_0$ weakly in $W^{1,q}(\mathbb{D})$, strongly in $L^q(\mathbb{D})$, and pointwise almost  everywhere. 
    
    Crucially, recalling property \emph{(d)}, the sequence $v_n$ is uniformly bounded. This ensures that the exponential terms $e^{v_n}$ and $e^{v_n/2}$ are uniformly bounded in $L^\infty(\mathbb{D})$. Then, by the Dominated Convergence Theorem, $e^{v_n} \to e^{v_0}$ in $L^p(\mathbb{D})$ for any $p \geq 1$. Passing to the limit, $v_0$ satisfies:
    \begin{equation*}
        \begin{cases}
            - \Delta v_0 = 2 K(p) e^{v_0} & \text{in } \mathbb{D}, \\
            \frac{\partial v_0}{\partial \nu} + 2 = 2 h(p) e^{v_0/2} & \text{on } \partial \mathbb{D}.
        \end{cases}
    \end{equation*}
    Furthermore, the barycenter condition implies $\int_{\mathbb{D}} x_1 e^{v_0}  = \int_{\mathbb{D}} x_2 e^{v_0} = 0$, which identifies $v_0$, following \cite[Lemma 2.2]{JevnikarLopezSorianoMedinaRuiz2022}, as:
    $$ v_0(x) = 2\log\frac{2\phi(p)}{\phi^2(p)+K(p)|x|^2}. $$

    Motivated by this limiting profile, we introduce the auxiliary sequence
    $$\tilde v_n(x) =2\log\frac{2\hat\phi_n}{\hat\phi^2_n+\hat k_n|x|^2},$$
    which corresponds to the exact solution of the Liouville problem with frozen coefficients (defined in \eqref{eq:def_hat_phi_and_k}) at the concentration point:
    \begin{equation*}
        \begin{cases}
            -\Delta \tilde{v}_n = 2\hat{k}_n e^{\tilde{v}_n} & \text{in } \mathbb{D}, \\
            \frac{\partial \tilde{v}_n}{\partial \nu} + 2 = 2 \hat{h}_n e^{\tilde{v}_n/2} & \text{on } \partial\mathbb{D}.
        \end{cases}
    \end{equation*}
    Observe that $\tilde v_n \to v_0$ in $C^k(\overline{\mathbb{D}})$ for any $k \in \mathbb{N}$. We also introduce the singular correction $s_n(x)$ as the unique solution with zero mean to:
    \begin{equation*}
        \begin{cases}
            - \Delta s_n = - 2 \tilde K_n(|f'_{a_n}(x)|^2-1) & \text{in } \mathbb{D}, \\
            \frac{\partial s_n}{\partial \nu} = -2(\tilde h_n-1)(|f'_{a_n}(x)|-1) & \text{on } \partial \mathbb{D}.
        \end{cases}
    \end{equation*}
    Note that the compatibility condition for this Neumann problem is satisfied because the right-hand sides integrate to zero. Additionally, observe that, by \eqref{mass}, standard elliptic estimates yield the quantitative bound 
    \begin{equation}
        \label{eq:norm_sn}
        \|s_n\|_{W^{1,q}} \leq C |\varepsilon_n|,
    \end{equation} 
    which implies $s_n \to 0$ strongly in $W^{1,q}(\mathbb{D})$ for any $q<2$.

    \medskip

    Defining the error term as $\xi_n(x) =v_n(x) - (\tilde v_n(x) + s_n(x))$, it follows from the convergence results above that $\xi_n \to 0$ in $W^{1,q}(\mathbb{D})$. Moreover, a direct computation shows that $\xi_n$ satisfies the linearized equation:
    \begin{equation}\label{eq:linearized_equation}
        \begin{cases}
            - \Delta \xi_n = 2 \hat{k}_n e^{\tilde v_n}\xi_n + c_n(x) & \text{in } \mathbb{D}, \\
            \frac{\partial \xi_n}{\partial \nu} = \hat{h}_n e^{\tilde v_n/2}\xi_n + d_n(x) & \text{on } \partial \mathbb{D},
        \end{cases}
    \end{equation}
    where the residuals are given by:
    \begin{equation*}
        \begin{split}
            c_n(x) &=2\left[K_n(f_{a_n}(x)) - \hat{k}_n\right]e^{v_n} + 2\hat{k}_n e^{\tilde v_n}\left(e^{v_n - \tilde{v}_n}-1-\xi_n\right) - 2 \tilde{K}_n, \\
            d_n(x) &=2\left[h_n(f_{a_n}(x)) - \hat{h}_n\right]e^{v_n/2}  + 2\hat{h}_n e^{\tilde v_n/2}\left(e^{(v_n - \tilde{v}_n)/2}-1-\frac{\xi_n}{2}\right) - 2(\tilde{h}_n-1).
        \end{split}
    \end{equation*}
    
  The objective is to estimate the norm $\|\xi_n\|_{W^{1+1/q,q}}$. To do so, we decompose the error as $\xi_n = \hat{\xi}_n + \sum_{j=1}^2 c_{n,j} Z_{n,j}$, where the component $\hat{\xi}_n$ satisfies:
    \begin{equation}\label{eq:ortogonal_condition}
        \int_{\mathbb{D}} x_j \hat{\xi}_n e^{\tilde{v}_n} \,  = 0, \quad \text{for } j=1,2,
    \end{equation}
    and $Z_{n,j}$ are defined by:
    $$ Z_{n,j}(x) =\frac{x_j}{\hat{\phi}_n^2 + \hat{k}_n|x|^2}, \quad \text{for } j=1,2. $$
    See \cite[Lemma 2.3]{JevnikarLopezSorianoMedinaRuiz2022} for further details. 

    Our strategy consists of analyzing $\hat{\xi}_n$ and $\sum_{j=1}^2 c_{n,j} Z_{n,j}$ separately. First, since $\hat{\xi}_n$ solves \eqref{eq:linearized_equation} and satisfies the orthogonality condition \eqref{eq:ortogonal_condition}, the linear estimates from \cite[Lemma 2.5]{JevnikarLopezSorianoMedinaRuiz2022} apply, yielding: 
    \begin{equation*}\label{eq:est_orth_pre}
        \|\hat{\xi}_n\|_{W^{1+1/q,q}} \leq C (\|c_n\|_{L^q} + \|d_n\|_{L^q}).
    \end{equation*}

    To estimate the norms on the right-hand side, we first observe that 
    $$ \|\tilde{K}_n\|_{L^\infty} + \|\tilde{h}_n-1\|_{L^\infty} \leq C |\varepsilon_n|. $$ 
    Furthermore, following the computations in \cite[Proposition 4.1 (e)]{JevnikarLopezSorianoMedinaRuiz2022}, we obtain:
    $$\left\|2\left[K_n(f_{a_n}(x)) - \hat{k}_n\right]e^{v_n} + 2\hat{k}_n e^{\tilde v_n}\left(e^{v_n - \tilde{v}_n}-1-\xi_n\right)\right\|_{L^q}\leq C(1-|a_n|)^{2/q},$$
    and
    $$\left\|2\left[h_n(f_{a_n}(x)) - \hat{h}_n\right]e^{v_n/2}  + 2\hat{h}_n e^{\tilde v_n/2}\left(e^{(v_n - \tilde{v}_n)/2}-1-\frac{\xi_n}{2}\right)\right\|_{L^q}\leq C(1-|a_n|)^{1/q}.$$
    
    Combining these estimates allows us to conclude:
    \begin{equation}\label{eq:est_orth}
        \|\hat{\xi}_n\|_{W^{1+1/q,q}} \leq C (1-|a_n|)^{1/q}+|\varepsilon_n|.
    \end{equation}
    
    In order to control the coefficients $c_{n,j}$, we expand the barycenter condition \emph{(b)} for each coordinate $x_k$, with $k \in \{1,2\}$: 
    \begin{equation}\label{eq:barycenter_expanded}
        \begin{split}
        0 & = \int_{\mathbb{D}} x_k e^{\tilde{v}_n} e^{s_n + \xi_n} \\
          & = \int_{\mathbb{D}} x_k e^{\tilde{v}_n} \left( 1 + s_n + \hat{\xi}_n + \sum_{j=1}^2 c_{n,j} Z_{n,j} + O((|s_n| + |\xi_n|)^2) \right).
        \end{split}
    \end{equation}
    Here we have performed a Taylor expansion of $e^{s_n+\xi_n}$. The validity of this expansion relies on the fact that the exponent $s_n + \xi_n = v_n - \tilde{v}_n$ is uniformly bounded in $L^\infty(\mathbb{D})$. 
    
    In equation \eqref{eq:barycenter_expanded} we have that $\int x_k e^{\tilde{v}_n} = 0$ by symmetry. Moreover, thanks to the orthogonality condition imposed on $\hat{\xi}_n$, $\int_{\mathbb{D}} x_k e^{\tilde{v}_n} \hat{\xi}_n =0$. Finally, $\int_{\mathbb{D}} x_k e^{\tilde{v}_n} Z_{n,j} =0$ if $j \neq k$, again by symmetry.  Consequently, we are led with:
    \begin{equation}\label{eq:decouple_equation}
        c_{n,k} \int_{\mathbb{D}} x_k Z_{n,k} e^{\tilde{v}_n}  =  - \int_{\mathbb{D}} x_k s_n e^{\tilde{v}_n} + O(\|\xi_n\|_{L^2}^2 + \varepsilon_n^2),\quad\text{for } k \in \{1,2\}.
    \end{equation}
    
    Substituting the explicit expression of $Z_{n,k}$, the integral coefficient on the left-hand side satisfies the nondegeneracy condition:
    \begin{equation}\label{eq:nondegeneracy}
        \left| \int_{\mathbb{D}} x_k Z_{n,k} e^{\tilde{v}_n} \right| = \int_{\mathbb{D}} \frac{x_k^2}{\hat{\phi}_n^2 + \hat{k}_n|x|^2} e^{\tilde{v}_n}  \geq C > 0,\quad\text{for } k \in \{1,2\}.
    \end{equation} 
    On the other hand, since $x_k e^{\tilde{v}_n}$ is uniformly bounded, we can estimate:
    $$ \left| \int_{\mathbb{D}} x_k s_n e^{\tilde{v}_n}  \right| \leq C \|s_n\|_{L^1} \leq C |\varepsilon_n|. $$
    By \eqref{eq:decouple_equation} and \eqref{eq:nondegeneracy}, we can isolate $c_{n,k}$ to obtain:
    \begin{equation*}
        |c_{n,k}| \leq C \left( |\varepsilon_n| + \|\xi_n\|_{L^2}^2 \right).
    \end{equation*}

    Combining these bounds with \eqref{eq:est_orth} we obtain:
    \begin{equation*}
        \|\xi_n\|_{W^{1+1/q,q}} \leq \|\hat{\xi}_n\|_{W^{1+1/q,q}} + C \sum_{j=1}^2 |c_{n,j}|  \leq C \left( (1-|a_n|)^{1/q} + |\varepsilon_n|+ \|\xi_n\|_{L^2}^2 \right).
    \end{equation*}
    
    Using the embedding $\|\xi_n\|_{L^2} \leq C \|\xi_n\|_{W^{1+1/q,q}}$ together with the fact that $\|\xi_n\|_{L^2} \to 0$, we get:
    $$ \|\xi_n\|_{W^{1+1/q,q}} \leq C \left[ (1-|a_n|)^{1/q} + |\varepsilon_n| \right]. $$

    By the Sobolev embedding $W^{1+1/q,q}(\mathbb{D}) \hookrightarrow C^{0,\alpha}(\overline{\mathbb{D}})$ with $\alpha=1-1/q$, the error term satisfies:
    \begin{equation}\label{eq:Holder_Xi}
        \|\xi_n\|_{C^{0,\alpha}} \leq C\left[(1-|a_n|)^{1-\alpha} + |\varepsilon_n|\right],
    \end{equation} 
    valid for any $\alpha \in (0, 1/2)$.
    
   Finally, we transfer the estimates \eqref{eq:norm_sn} and \eqref{eq:Holder_Xi} to the original coordinates via $\psi_n = (\xi_n + s_n) \circ f_{-a_n}$. First, the correction term $\tilde{s}_n =s_n \circ f_{-a_n}$ satisfies:
    \begin{equation*}
    \begin{cases}
        -\Delta \tilde{s}_n(x) = -2\tilde{K}_n \left( 1 - |f'_{-a_n}(x)|^2 \right) & \text{in } \mathbb{D}, \\
        \frac{\partial \tilde{s}_n}{\partial \nu}(x) = -2(\tilde{h}_n-1) \left( 1 - |f'_{-a_n}(x)| \right) & \text{on } \partial\mathbb{D}.
    \end{cases}
    \end{equation*}
    Again, by the Sobolev embedding $W^{1+1/q,q}(\mathbb{D}) \hookrightarrow C^{0,\alpha}(\overline{\mathbb{D}})$ with $\alpha=1-1/q$, combined with the linear estimates from \cite[Lemma 2.4]{JevnikarLopezSorianoMedinaRuiz2022}, we have:
    \begin{equation*}
        \|\tilde{s}_n\|_{C^{0,\alpha}} \leq C \|\tilde{s}_n\|_{W^{1+1/q,q}} \leq C |\varepsilon_n|\left( 1 + \left\| |f'_{-a_n}|^2 \right\|_{L^q(\mathbb{D})} + \|f'_{-a_n}\|_{L^q(\partial\mathbb{D})}\right).
    \end{equation*}
    Observe that $\|f_{a_n}'\|_{L^{\infty}} = C (1-|a_n|)^{-1}$. Now, interpolating between the $L^1$ and $L^\infty$ norms we get:
    \begin{equation*}
    \begin{split}
        \left\| |f'_{-a_n}|^2 \right\|_{L^q} 
        &\leq \left\| |f'_{-a_n}|^2 \right\|_{L^1}^{1/q} \cdot \left\| |f'_{-a_n}|^2 \right\|_{L^\infty}^{1-1/q} \\
        &\leq C (1-|a_n|)^{-2\alpha}.
    \end{split}
    \end{equation*}
    Similarly $\|f'_{-a_n}\|_{L^q(\partial\mathbb{D})}\leq C (1-|a_n|)^{-\alpha}$, so we conclude:
    \begin{equation}\label{eq:tilde_sn_estimate}
        \|\tilde{s}_n\|_{C^{0,\alpha}} \leq C |\varepsilon_n| (1-|a_n|)^{-2\alpha}.
    \end{equation}

    In contrast, for $\tilde{\xi}_n=\xi_n \circ f_{-a_n}$ we can use \eqref{eq:Holder_Xi} to derive:
    \begin{equation*}
        [\tilde{\xi}_n]_{C^{0,\alpha}}\leq [\xi_n]_{C^{0,\alpha}} \, \|f'_{-a_n}\|_{L^\infty}^\alpha \leq C \|\xi_n\|_{C^{0,\alpha}} (1-|a_n|)^{-\alpha}.
    \end{equation*}
    
    Combining these bounds we can conclude 
    $$[\psi_n]_{C^{0,\alpha}} \leq C (1-|a_n|)^{-2\alpha} \left[ (1-|a_n|) + |\varepsilon_n| \right].$$ 
    
    To conclude the proof of \eqref{eq:norm_psi}, we only need to control the norm $\|\psi_n\|_{C^0}$, which can be bounded by $\|\psi_n\|_{C^0} \leq \|\xi_n\|_{C^0} + \|\tilde{s}_n\|_{C^{0}}$.
    From \eqref{eq:Holder_Xi} and \eqref{eq:tilde_sn_estimate}, we conclude that: $$\|\psi_n\|_{C^0} \leq C[(1-|a_n|)^{1-\alpha} + |\varepsilon_n|].$$
    
    Therefore, the estimate for $[\psi_n]_{C^{0,\alpha}}$ controls the full norm, yielding the desired result.
\end{proof}

Once we have an appropriate estimate on the error term $\psi_n$, we now plan to use Kazdan-Warner identities to obtain information on the derivatives of $\Phi$ at the blow-up point $p$. The proof of Theorem \ref{thm:normal_condition} relies on the following three lemmas.
       
\begin{lemma}
$\partial_\tau \Phi(p)= 0$.
\end{lemma}

\begin{proof}
Define the vector field $ F : \overline{\D} \to \mathbb{R}^2 $ by $ F(x_1, x_2) = (-x_2, x_1) $. Observe that $ F $ is tangential to $ \partial \D $ and oriented in the same direction as the tangential vector $ \tau $.

Applying Lemma \ref{lem:pohosaev}, and using the identities $ \nabla \cdot F = 0 $ and $ DF(\nabla u_n, \nabla u_n) = 0 $, we obtain the following expression:
$$
    \int_{\partial \D} \left(h_n e^{u_n/2} - \tilde{h}_n\right) \partial_\tau u_n 
    = \int_\D \left[ \tilde{K}_n \nabla u_n \cdot F + e^{u_n} \nabla K_n \cdot F \right].
$$

Integrating by parts and rearranging terms yields
\begin{equation}\label{eq:tangential_pohosaev}
    -\int_{\partial \D} \tilde{h}_n \, \partial_\tau u_n 
    = \int_{\partial \D} 2\, \partial_\tau h_n \, e^{u_n/2} 
    + \int_\D \left( \tilde{K}_n \nabla u_n \cdot F + e^{u_n} \nabla K_n \cdot F \right).
\end{equation}

From Proposition 3.1 in \cite{JevnikarLopezSorianoMedinaRuiz2022}, the following 

convergences hold:
$$
\begin{aligned}
    u_n &\to -\infty, \quad \text{uniformly in compact sets of } \overline{\D} \setminus \{p\}, \\
    h_n e^{u_n/2} &\rightharpoonup 2\pi \frac{h(p)}{\sqrt{h^2(p)+K(p)}} \delta_p,\quad \text{in the sense of measure supported on } \partial\D,\\
    K_n e^{u_n} &\rightharpoonup 2\pi \left(1 - \frac{h(p)}{\sqrt{h^2(p)+K(p)}}\right) \delta_p, \quad\text{in the sense of measure supported on } \D.
\end{aligned}
$$

Since $\tilde{K}_n$ and $\tilde{h}_n$ are constant functions, we obtain:
$$
    \tilde{K}_n \int_{\D} \nabla u_n \cdot F 
    = -\tilde{K}_n \int_\D u_n \nabla \cdot F = 0, 
    \quad 
    \tilde{h}_n \int_{\partial \D} \partial_\tau u_n = 0.
$$

Substituting these limits into \eqref{eq:tangential_pohosaev} and passing to the limit as $n \to \infty$, we obtain:
$$
    \frac{4\pi \, \partial_\tau h(p)}{\sqrt{h^2(p)+K(p)}} 
    + \frac{2\pi \, \partial_\tau K(p)}{\Phi(p) \sqrt{h^2(p)+K(p)}} = 0,
$$
which is equivalent to $\partial_\tau \Phi(p) = 0$.
\end{proof}

The estimate of the normal derivative of $\Phi$ requires a more accurate description of the behavior of $u_n$. By using convenient rotations we can assume, without loss of generality, that $p=(1,0)$ and $a_n\in\R$ and $a_n\to1$. Then the profile given in Theorem \ref{thm: Blowup} reduces to 
    \begin{equation}\label{eq:profile}
        u_{ a_n}(x_1,x_2):=2\log\left\{\frac{2(1- a_n^2)\hat{\phi}_n}{\hat{\phi}_n^2(1- a_n x_1)^2+\hat{\phi}_n^2 a_n^2x_2^2+\hat{k}_n(x_1- a_n)^2+\hat{k}_nx_2^2}\right\}.
    \end{equation}

\begin{lemma}\label{lem:descomposition}
    If $u_n = u_{ a_n} + \psi_n$, where $u_{ a_n}$ is given by \eqref{eq:profile}, then:
    \begin{equation}\label{eq:old_estimates_1}
        \int_{\partial\D}\partial_\tau h_n e^{u_n/2} x_2 =e^{\psi_n(p)/2} \int_{\partial\D} \partial_\tau h_n e^{u_{ a_n}/2} x_2 + O(1- a_n)^{1-2\alpha}\left(1- a_n+|\varepsilon_n|\right),
    \end{equation}
    \begin{equation}\label{eq:old_estimates_2}
    \begin{split}
        \int_\D \left( \partial_{x_2} K_n\, x_1x_2 - \partial_{x_1} K_n (1 - x_1^2 + x_2^2) \right)e^{u_n} =  & \\ e^{\psi_n(p)} \int_\D \left( \partial_{x_2} K_n\, x_1x_2 - \partial_{x_1} K_n (1 - x_1^2 + x_2^2) \right) e^{u_{ a_n}} 
        & + O(1- a_n)^{1-2\alpha}\left(1- a_n+|\varepsilon_n|\right),
    \end{split}
    \end{equation}
    \begin{equation*}\label{eq:new_estimates_1}
        \int_{\partial\D} x_1 u_n (\tilde{h}_n - 1) =\int_{\partial\D} x_1 u_{ a_n} (\tilde{h}_n - 1) + O(1- a_n)^{1-2\alpha}\left(1- a_n+|\varepsilon_n|\right),
    \end{equation*}
    and    
    \begin{equation*}\label{eq:new_estimates_2}
        \int_{\D} x_1 u_n \tilde{K}_n =\int_{\D} x_1 u_{ a_n} \tilde{K}_n +O(1- a_n)^{1-2\alpha}\left(1- a_n+|\varepsilon_n|\right).
    \end{equation*}
\end{lemma}

\begin{proof}
    The estimates in \eqref{eq:old_estimates_1} and \eqref{eq:old_estimates_2} can be obtained following the same computations of \cite[Lemma 5.2]{JevnikarLopezSorianoMedinaRuiz2022} and making use of the $C^{0,\alpha}$ bound for $\psi_n$ provided in \eqref{eq:norm_psi}. The last two estimates are also immediate consequences of this bound.
\end{proof}

The following result completes the proof of Theorem \ref{thm:normal_condition}.
        
\begin{lemma}\label{lem:sign_relation}  If  $\e_n>0$, then $  \partial_\nu \Phi(p) \geq 0$. Instead, if $\e_n<0$, then $\partial_\nu \Phi(p) \leq 0$
\end{lemma}

\begin{proof}

    We start from the integral identity derived in Proposition \ref{prop:Kazdan-Warner_ind}:
    \begin{equation}\label{eq:KW_identity}
        \int_{\partial\mathbb{D}}\left[4x_2\partial_\tau h_n e^{u_n/2} - 2x_1 u_n(\tilde{h}_n - 1)\right] = \int_\mathbb{D} \left[4x_1 u_n \tilde{K}_n + e^{u_n}\nabla K_n \cdot F\right],
    \end{equation}
    where $F(x_1,x_2) = (1-x_1^2 + x_2^2, -2x_1x_2)$. Applying the expansions obtained in Lemma \ref{lem:descomposition} to each term and grouping the integrals involving the profile $u_{a_n}$, equation \eqref{eq:KW_identity} becomes:
    \begin{equation}\label{eq:decomp}
        \begin{split}
            4 e^{\psi_n(p)/2} & \int_{\partial\mathbb{D}} \partial_\tau h_n e^{u_{a_n}/2} x_2 - e^{\psi_n(p)} \int_\mathbb{D} e^{u_{a_n}} \nabla K_n \cdot F \\
            &= 4\tilde{K}_n \int_{\mathbb{D}} x_1 u_{a_n} + 2(\tilde{h}_n - 1) \int_{\partial\mathbb{D}} x_1 u_{a_n} + \mathcal{R}_n,
        \end{split}
    \end{equation}
    where the remainder satisfies $|\mathcal{R}_n| \leq C(1- a_n)^{1-2\alpha}\left[1- a_n + |\varepsilon_n|\right]$.

    We analyze the asymptotic behavior of the leading terms. For the right-hand side of \eqref{eq:decomp}, as justified in Appendix \ref{sec:Appendix}, we have the convergence:
    \begin{equation*}
        4\tilde{K}_n \int_{\mathbb{D}} x_1 u_{a_n} + 2(\tilde{h}_n - 1)\int_{\partial\mathbb{D}} x_1 u_{a_n} = - \frac{10\pi \varepsilon_n}{1+\varepsilon_n}(1+o(1)).
    \end{equation*}

   For the left-hand side, following the computations in \cite[Prop. 5.3]{JevnikarLopezSorianoMedinaRuiz2022} and identifying the boundary term resulting with the half-Laplacian (see \cite[Appendix]{DaLioMartinazziRiviere2015}), we obtain:
    \begin{equation*}
    \begin{split}
        \int_{\partial\mathbb{D}} \partial_\tau h_n e^{u_{a_n}/2}x_2 &= 2\pi(1- a_n^2)\frac{\hat{\phi}_n}{\hat{\phi}_n^2+\hat{k}_n} (-\Delta)^{1/2} h(p) + o(1- a_n),\\
        \int_\mathbb{D} e^{u_{a_n}} \nabla K_n \cdot F &= -2\pi(1- a_n^2)\frac{\partial_{\nu}K(p)}{\hat{\phi}_n^2+\hat{k}_n} + o(1- a_n).
    \end{split}
    \end{equation*}
    
    Substituting these expansions into \eqref{eq:decomp}, we get:
    \begin{equation}\label{eq:decomp_2}
    \begin{split}
        \frac{2\pi(1- a_n^2)}{\hat{\phi}_n^2+\hat{k}_n} & \Bigg( 4 \hat{\phi}_n e^{\psi_n(p)/2} (-\Delta)^{1/2} h(p) + e^{\psi_n(p)} \partial_{\nu}K(p) \Bigg) + o(1-a_n) \\
        &= - \frac{10\pi \varepsilon_n}{1+\varepsilon_n}(1+o(1)) + \mathcal{R}_n.
    \end{split}
\end{equation}

    Since $\psi_n$ is bounded, we obtain  $|\varepsilon_n| \leq C (1- a_n)$. With this information we come back to the norm estimate \eqref{eq:norm_psi}:
    \begin{equation*}
        \|\psi_n\|_{C^{0,\alpha}} \leq C (1- a_n)^{-2\alpha} \left(1- a_n+|\varepsilon_n| \right) \leq C (1- a_n)^{1-2\alpha}.
    \end{equation*}
    Since $\alpha < 1/2$, this implies $\|\psi_n\|_{L^\infty} \to 0$ as $n \to \infty$, and in particular $e^{\psi_n(p)} \to 1$. Using again \eqref{eq:decomp_2} and recalling the definition of $\Phi$  we have: 
    $$ \sqrt{h(p)^2+K(p)} \, \partial_\nu \Phi(p) = \Phi(p) (-\Delta)^{1/2} h(p) + \frac{1}{2} \partial_\nu K(p). $$
    Thus, the equation \eqref{eq:decomp_2} simplifies to:
    \begin{equation*}
        - \frac{\hat{\phi}_n}{\hat{\phi}_n^2+\hat{k}_n} (1- a_n^2) \partial_\nu \Phi(p) = - 10\pi \varepsilon_n + o(1- a_n).
    \end{equation*}
    
    This concludes the proof.
\end{proof}
%%%%%%%%%%%%%%%%%%%%%%%%%%%%%%%%%%%%%%%%%%%%%%%%%%%%%%%%%%%%%%%%%%%%%%%%%%%%%%%%%%%%%%%%%%%%%%%%%%%%%%%%%%%%%%%%%%%%%%%%%%%%%%%%%%%%%%%%%%%%%%%%%%%%%%%%%%%%%%%%%%%%%%%%%%%%%%%%%%%%%%%%%%%%%%%%%%%%%%%% Section 5 %%%%%%%%%%%%%%%%%%%%%%%%%% %%%%%%%%%%%%%%%%%%%%%%%%%%%%%%%%%%%%%%%%%%%%%%%%%%%%%%%%%%%%%%%%%%%%%%%%%%%%%%%%%%%%%%%%%%%%%%%%%%%%%%%%%%%%%%%%%%%%%%%%%%%%%%%%%%%%%%%%%%%%%%%%%%%%%%%%%%%%%%%%%%%%%%%%%%%%%%

\section{Proof of the Main Theorems}\label{sec:proofs}

\begin{proof}[Proof of Theorems \ref{thm:existence_mountain_pass} and \ref{thm:existence_linking}]
    The proofs of both theorems follow a unified strategy. Let $(u_n)$ be the sequence of solutions to the perturbed problems obtained in Theorem \ref{prop:ConvergenceSubsequence} (with $\e_n>0$) or in Theorem  \ref{prop:ConvergenceSubsequence2} (with $\e_n<0$). Recall that, by construction, this sequence has uniformly bounded Morse index. We claim that it is also uniformly bounded from above. Once this uniform bound is established, standard elliptic regularity and compactness arguments guarantee that, up to a subsequence, $u_n$ converges to a solution of the original problem \eqref{eq:Problem}.

    Reasoning by contradiction, assume that $\max_{\overline{\mathbb{D}}} u_n \to +\infty$. According to Theorem \ref{thm: Blowup}, the singular set $S$ decomposes as $S_0 \cup S_1$. Moreover,
    \begin{itemize}
        \item In Theorem \ref{thm:existence_mountain_pass}, condition (ii) explicitly assumes that $\partial_\tau \mathfrak{D}(x) \neq 0$ whenever $\mathfrak{D}(x) = 1$. This implies that $S_0 = \emptyset$.
        \item In Theorem \ref{thm:existence_linking}, condition (i) implies that $\mathfrak{D}(x) > 1$ everywhere on $\partial \mathbb{D}$, so again $S_0 = \emptyset$.
    \end{itemize}
    
    Therefore, in virtue of Theorem \ref{thm: Blowup}, the singular set consists of a single boundary point $S=\{p\}$. Applying Theorem \ref{thm:normal_condition}, we have that $\partial_\tau \Phi(p)= 0$ and the sign relation:
    $$
    \e_n \, \partial_{\nu} \Phi(p)\geq 0
    $$
    This contradicts the geometric hypotheses in both cases: condition (iii) of Theorem \ref{thm:existence_mountain_pass} (if we choose $\varepsilon_n > 0$) and condition (ii) of Theorem \ref{thm:existence_linking} (if we choose $\varepsilon_n < 0$). 
    
    \end{proof}

%%%%%%%%%%%%%%%%%%%%%%%%%%%%%%%%%%%%%%%%%%%%%%%%%%%%%%%%%%%%%%%%%%%%%%%%%%%%%%%%%%%%%%%%%%%%%%%%%%%%%%%%%%%%%%%%%%%%%%%%%%%%%%%%%%%%%%%%%%%%%%%%%%%%%%%%%%%%%%%%%%%%%%%%%%%%%%%%%%%%%%%%%%%%%%%%%%%%%%%% Appendix %%%%%%%%%%%%%%%%%%%%%%%%%% %%%%%%%%%%%%%%%%%%%%%%%%%%%%%%%%%%%%%%%%%%%%%%%%%%%%%%%%%%%%%%%%%%%%%%%%%%%%%%%%%%%%%%%%%%%%%%%%%%%%%%%%%%%%%%%%%%%%%%%%%%%%%%%%%%%%%%%%%%%%%%%%%%%%%%%%%%%%%%%%%%%%%%%%%%%%%%

\appendix
\numberwithin{equation}{section}
\section{} \label{sec:Appendix}
This appendix is devoted to the computation of the integral limits used in the proof of Theorem \ref{thm:normal_condition}.

\begin{lemma}
    Let $u_{a_n}$ be the profile defined in \eqref{eq:profile}, where $a_n\in(0,1)$ and $a_n\to1$. Then
    \begin{equation*}
        \lim_{n\to+\infty}\int_{\mathbb{D}} x_1 u_{a_n} = \pi \quad \text{and} \quad \lim_{n\to+\infty}\int_{\partial\mathbb{D}} x_1 u_{a_n} = 4\pi.
    \end{equation*}
\end{lemma}

\begin{proof}
   Using the explicit expression of the profile, we rewrite the integrand as:
    \begin{equation*}
            \int_\mathbb{D} x_1 u_{a_n} = -2\int_\mathbb{D} x_1 \log{\left[\hat{\phi}_n^2|1- a_n x|^2+\hat{k}_n|x- a_n|^2\right]}.
    \end{equation*}
    Define the function $\tilde{u}_{a_n}(x) =-2\log{\left[\hat{\phi}_n^2|1- a_n x|^2+\hat{k}_n|x- a_n|^2\right]}$. As $n \to \infty$, the coefficients satisfy $\hat{\phi}_n^2 + \hat{k}_n \to \delta$, where $\delta =\hat{\phi}_0^2+\hat{k}_0$.
    
    Consequently,
    $$\tilde{u}_{a_n}(x) \to -2\log{\left(\delta|x-1|^2\right)},$$ 
    pointwise for all $x \in \mathbb{D} \setminus \{1\}$.
    
    To apply the Dominated Convergence Theorem, we need to find an integrable majorant. Fix $n_0 \in \mathbb{N}$ large enough such that for all $n \geq n_0$, $\hat{\phi}_n^2 + \hat{k}_n \geq \overline{\delta} > 0$ for some constant $\overline{\delta} < \delta$. Noticing that $\hat{k}_n$ is negative, we can write $\hat{\phi}_n^2 \geq \overline{\delta} + |\hat{k}_n|$. Combining this with the inequality $|1-a_n x|^2 \geq |x-a_n|^2$ (valid for $|x|\leq1$), we observe:
    \begin{equation*}
        \begin{split}
            -2\log{\left[\hat{\phi}_n^2|1- a_n x|^2+\hat{k}_n|x- a_n|^2\right]} &\leq -2\log{\left[\overline{\delta}|1- a_n x|^2 + |\hat{k}_n|( |1-a_n x|^2-|x-a_n|^2)\right]} \\
             &\leq -2\log{\left[\overline{\delta}|1- a_n x|^2\right]}.
        \end{split}
    \end{equation*}

    Furthermore, employing the estimate $|1-a_n x|^2 \geq a_n|x-1|^2$ (valid since $|x| \leq 1$ and $a_n \in (0,1)$), we obtain:
    \begin{equation*}
        \begin{split}
            -2\log{\left[\overline{\delta}|1- a_n x|^2\right]} &\leq -2\log{\left[a_n \overline{\delta} |1- x|^2\right]} \\
            &\leq -2\log{\left[\frac{\overline{\delta}}{2} |1- x|^2\right]},
        \end{split}
    \end{equation*}
    where we assumed $a_n \geq 1/2$ for sufficiently large $n$. Since the function $\left|\log|1-x|\right|$ is integrable over $\mathbb{D}$, this majorant allows us to pass the limit inside the integral.
    
    We proceed to compute the limit integral. By symmetry, we have:
    \begin{equation*}
        \begin{split}
           -2\int_\mathbb{D} x_1  \left[ \log \delta + \log (|1- x|^2) \right] &= -2\int_\mathbb{D} x_1 \log{|1- x|^2}\\
           &=\int_0^1\rho^2 \left( \int_0^{2\pi}\cos\theta \log(1-2\rho\cos\theta+\rho^2)d\theta \right) d\rho= \pi,
        \end{split}
    \end{equation*}
    where we have used the identity from \cite[Equation 4.224(9)]{GradshteynRyzhik2007}:
    $$\int_0^{2\pi}\cos (\theta) \log(1-2\rho\cos\theta+\rho^2)d\theta = -2\pi \rho \quad (\text{for } \rho^2 < 1).$$
    
    \bigskip
    The boundary term is analyzed analogously. Following the same dominance argument established above, we obtain:
    $$ \lim_{n\to \infty} \int_{\partial\mathbb{D}} x_1 u_{a_n} = -2 \int_{\partial\mathbb{D}} x_1 \log|1-x|^2 = -2 \int_0^{2\pi} \cos\theta \log(2-2\cos\theta) \, d\theta=4\pi.$$
\end{proof}

        \bibliographystyle{plain}
	\bibliography{references.bib} 
  \end{document}